\documentclass{article}

\usepackage[preprint]{neurips_2026}

\usepackage[utf8]{inputenc} 
\usepackage[T1]{fontenc}    

\usepackage{xcolor}  
\usepackage{url}            
\usepackage{booktabs}       
\usepackage{amsfonts}       
\usepackage{aliascnt} 
\usepackage{nicefrac}       
\usepackage{microtype}      
\usepackage{xcolor}         
\usepackage{amsmath}
\usepackage{amsthm}
\usepackage{algorithm}
\usepackage{algpseudocode}

\usepackage{overpic}
\usepackage{graphicx}
\usepackage{subcaption}

\usepackage{jabbrv}
\DefineSpuriousJournalWord{on}
\DefineSpuriousJournalWord{for}
\DefineSpuriousJournalWord{its}
\DefineJournalAbbreviation{Foundations}{Found}
\DefineJournalAbbreviation{Transactions}{Trans}
\DefineJournalAbbreviation{National}{Natl}
\DefineJournalAbbreviation{Quantification}{Quantif}

\usepackage{hyperref}
\usepackage{bookmark}
\hypersetup{breaklinks=true}

\usepackage[capitalize]{cleveref}

\theoremstyle{plain}
\newtheorem{theorem}{Theorem}[section]
\newaliascnt{lemma}{theorem}
\newtheorem{lemma}[lemma]{Lemma}
\aliascntresetthe{lemma}
\newaliascnt{proposition}{theorem}
\newtheorem{proposition}[proposition]{Proposition}
\aliascntresetthe{proposition}
\newaliascnt{corollary}{theorem}
\newtheorem{corollary}[corollary]{Corollary}
\aliascntresetthe{corollary}
\theoremstyle{definition}
\newaliascnt{definition}{theorem}
\newtheorem{definition}[definition]{Definition}
\aliascntresetthe{definition}
\newaliascnt{assumption}{theorem}

\aliascntresetthe{assumption}
\theoremstyle{remark}
\newaliascnt{remark}{theorem}
\newtheorem{remark}[remark]{Remark}
\aliascntresetthe{remark}
\newaliascnt{example}{theorem}

\aliascntresetthe{example}
\crefname{theorem}{Theorem}{Theorems}
\crefname{lemma}{Lemma}{Lemmas}
\crefname{proposition}{Proposition}{Propositions}
\crefname{corollary}{Corollary}{Corollaries}
\crefname{definition}{Definition}{Definitions}
\crefname{assumption}{Assumption}{Assumptions}
\crefname{remark}{Remark}{Remarks}
\crefname{example}{Example}{Examples}
\crefname{algorithm}{Algorithm}{Algorithms}

\title{Conformal Uncertainty Quantification Guarantees for Neural Operators}

\author{
Tom Stent \\
Department of Mathematics\\
Imperial College London\\
London, SW7 2AZ, UK \\
\texttt{tom.stent25@imperial.ac.uk}
\And
Nicolas Boullé \\
Department of Mathematics\\
Imperial College London\\
London, SW7 2AZ, UK \\
\texttt{n.boulle@imperial.ac.uk}
}

\begin{document}

\maketitle

\begin{abstract}
  Neural operators provide fast surrogate models for approximating operators between function spaces, but their predictions often lack uncertainty quantification. We develop a split conformal framework to guarantee that a calibrated pointwise band around the neural operator output contains the true solution on at least a $1-\gamma$ fraction of the evaluation domain, with probability at least $1-\alpha$ over test and calibration inputs, where $\alpha,\gamma \in(0,1)$. Our method reduces a normalized residual field to its spatial $(1-\gamma)$-quantile and computes a scaling factor using a held-out calibration dataset. We prove marginal coverage guarantees for measurable residual fields defined on arbitrary probability spaces, covering both  continuum domains and fixed discretizations. Under mild assumptions on the data distribution, we show that the coverage conditional on the calibration set follows a Beta distribution, which we verify with numerical experiments on Darcy flow and Navier–Stokes equations, where our calibration yields bands consistently tighter than existing corrections while retaining the target coverage.
\end{abstract}

\section{Introduction}

Operator learning~\citep{boulle2024mathematical,kovachki2023neural,kovachki2024operator,li2021fourier,lu2021deeponet} has become an increasingly popular approach within the field of scientific machine learning~\citep{karniadakis2021physics} to approximate nonlinear operators, such as solution operators associated with partial differential equations (PDEs), between infinite-dimensional function spaces using pairs of input-output functions. A wide range of neural operator architectures, generalizing neural networks to infinite dimensions, have been recently introduced such as DeepONet \citep{lu2021deeponet}, Fourier neural operator (FNO) \citep{kovachki2023neural,li2021fourier}, convolutional neural operators \citep{raonic2023convolutional}, U-shaped neural operators \citep{rahman2023uno}, multi-resolution architectures~\citep{bouziani2024structure,gupta2021multiwavelet,he2024mgno,li2020multipole}, transformer neural operators \citep{cao2021choose,hao2023gnot}, and foundation models \citep{herde2024poseidon,wang2026mixture}. These data-driven PDE surrogate techniques have also been applied to large-scale weather forecasting systems \citep{bi2023pangu,lam2023graphcast,pathak2022fourcastnet}.

While several theoretical works studied approximation and sample complexity properties of neural operators \citep{boulle2024operator,boulle2023elliptic,kovachki2021universal,kovachki2024operator,lanthaler2022error}, ensuring the reliability of their predictions remains a key challenge. In particular, quantifying uncertainty or estimating prediction errors in operator learning may be crucial for applications like weather forecasting, climate modelling, and engineering design. Several neural operator uncertainty quantification methods have been recently introduced \citep{psaros2023uncertainty,zou2024neuraluq}, such as Bayesian neural operators \citep{magnani2022abno}, function-valued Gaussian processes \citep{magnani2024luno}, and probabilistic neural operators trained with proper scoring rules \citep{bulte2025probabilistic}. Ensemble methods have also been used to identify high-error regions and improve operator predictions under distribution shift \citep{mouli2024using}. These approaches adapt broader Bayesian and ensemble techniques \citep{gal2016dropout,gawlikowski2023survey,he2026survey,lakshminarayanan2017simple}, but the resulting uncertainty estimates usually depend on modelling assumptions, training objectives, or data distribution, and generally lack coverage guarantees.

Conformal prediction is a technique that uses nonconformity scores computed on a held-out calibration set to construct a prediction set containing the ground truth solution with high probability \citep{angelopoulos2023gentle,angelopoulos2024conformalbook,shafer2008tutorial,vovk2005algorithmic}. One computes these scores on calibration inputs and uses an upper empirical quantile as the threshold; this yields finite sample, marginal coverage under exchangeability, i.e., coverage in probability over a new draw from the same data distribution. Recent scientific machine learning methods exploit conformal prediction cellwise, grouping nonconformity scores across pairs of calibration sample and grid location and calibrating a single quantile \citep{gopakumar2024valid,gopakumar2025calibrated,gopakumar2026uncertainty,moya2025conformalized,yu2026conformal}. One limitation of this approach is that the guarantee is marginal over both test solutions and spatial locations, hence controlling the average fraction of covered solution-location pairs, instead of the fraction of the domain covered for each individual solution. \citet{gray2025guaranteed,wang2026operator} construct simultaneous prediction sets for function-valued surrogate outputs to obtain strong spatial guarantees. However, the coverage band must accommodate the worst location on the grid, which may result in uninformative estimates for difficult problems. Then, \cite{millard2025split} extend split conformal prediction to neural operators by calibrating in a discretized function space and lifting the guarantee to the underlying function space, assuming that the discretization map is bilipschitz. Closer to this work is the UQNO framework introduced in \citep{ma2024calibrated}, which trains two neural operators: a predictor for the solution and an estimator for the error field, then calibrates a single scaling factor to guarantee that the ground truth is covered by a prediction band over at least a target fraction of the spatial domain. \cite{ma2024calibrated} employ a Hoeffding correction to lift empirical spatial coverage to the continuum domain. This correction requires i.i.d.\ spatial evaluation points and therefore does not apply to the deterministic grids used by standard FNO implementations and may overestimate the error.

\paragraph{Main contributions.}
This work rigorously formulates split conformal calibration for neural operators using a generic measurable residual field, modelling the ratio between the prediction error and the estimated one, over an arbitrary probability space. This framework applies to both the continuum and discretized settings, giving finite-sample marginal coverage under exchangeability of the data without concentration corrections. Moreover, under weak assumptions on the data distribution (i.i.d.\ inputs inducing a continuous distribution of conformal scores), we show that the coverage conditional on the calibration set follows a Beta distribution, which quantifies the variance of the marginal guarantee with respect to the calibration set size. Our numerical experiments on the two-dimensional Darcy flow and Navier--Stokes equations demonstrate the validity of our coverage guarantees.

\section{Calibration Set and Coverage Guarantees}\label{sec:theory}

\paragraph{Setup.}
Let $D \subset \mathbb{R}^d$ be a bounded, Lebesgue measurable, domain, with $0 < \mathrm{Leb}(D) < \infty$, and define the normalized Lebesgue measure $\nu (B) := \mathrm{Leb}(B) / \mathrm{Leb}(D)$. We consider separable Banach spaces $\mathcal{A}=\mathcal{A}(D; \mathbb{R}^{d_a})$ and $\mathcal{U}=\mathcal{U}(D; \mathbb{R}^{d_u})$, whose elements are measurable vector-valued input and output functions on $D$. We equip these spaces with their respective Borel sigma algebras $\mathcal{B}(\mathcal{A})$ and $\mathcal{B}(\mathcal{U})$, and consider the Borel measurable operators $\mathcal{G}^\dagger, \hat{\mathcal{G}} : \mathcal{A} \rightarrow \mathcal{U}$,
representing the true PDE solution operator and a pre-trained neural operator (approximating $\mathcal{G}^\dagger$). We study the Borel measurable error estimator $\hat{\mathcal{E}}: \mathcal{A} \rightarrow \mathcal{U}(D; \mathbb{R})$, which outputs a pointwise scalar estimate of the error of predictions of $\hat{\mathcal{G}}$, with $\hat{\mathcal{E}}(a) \ge 0$ $\nu$-almost everywhere (a.e.) for every $a \in \mathcal{A}$. We model $\hat{\mathcal{E}}$ as an FNO trained on the quantile loss of the error field of predictions of $\hat{\mathcal{G}}$ \citep{ma2024calibrated}, however all results apply for any pointwise error estimation technique satisfying the relevant assumptions. Note that FNO architectures consist of compositions of continuous operators, so are themselves continuous (and therefore measurable). Therefore, for
$\hat{\mathcal{G}}$ and $\hat{\mathcal{E}}$ represented as such, Borel measurability holds. Moreover, if the underlying PDE problem is Hadamard well-posed between the spaces $\mathcal{A}$ and $\mathcal{U}$, then $\mathcal{G}^\dagger : \mathcal{A} \to \mathcal{U}$ is continuous, hence Borel. We emphasize that continuity is only a sufficient condition as we only require Borel measurability of $\mathcal{G}^\dagger$, $\hat{\mathcal{G}}$ and $\hat{\mathcal{E}}$.

\paragraph{Motivation.} Let $a\in\mathcal{A}$ be an input function and consider the neural operator prediction $\hat{\mathcal{G}}(a) \approx \mathcal{G}^\dagger(a)\in \mathcal{U}$ and the error estimator $\hat{\mathcal{E}}(a) \approx \|\mathcal{G}^\dagger(a) - \hat{\mathcal{G}}(a)\|_2$. Given $x\in D$, we aim to construct a closed ball prediction band $C_\lambda(a)(x) = \bar{B}(\hat{\mathcal{G}}(a)(x), \lambda \hat{\mathcal{E}}(a)(x))\subset \mathbb{R}^{d_u}$ around the prediction $\hat{\mathcal{G}}(a)(x)$, with a scaling factor $\lambda\in [0,\infty]$ controlling the radius of the ball. The goal is to choose $\lambda$ such that the true solution $\mathcal{G}^\dagger(a)(x)$ lies within this band for a fraction $1-\gamma$ of points $x\in D$ (with respect to the measure $\nu$), with probability greater than $1-\alpha$ over the draw of $a$ from a distribution $\mu$ on $\mathcal{A}$, where $\alpha,\gamma\in(0,1)$ are fixed.

\subsection{Calibration set and scaling factor}\label{sec:setup}

Since $\mathcal{G}^\dagger(a)$, $\hat{\mathcal{G}}(a)$ and $\hat{\mathcal{E}}(a)$ are only defined up to equivalence classes in $\mathcal{U}$, the containment set
$\{x \in D : \mathcal{G}^\dagger(a)(x) \in C_\lambda(a)(x)\}$ is not specified pointwise. We must therefore check both that its $\nu$-measure is unchanged by the choice of representatives, and that this measure is a Borel function of $a$. It suffices to assume the existence of product-measurable representatives $g^\dagger, \hat{g}:\mathcal{A}\times D \rightarrow \mathbb{R}^{d_u}$, $\hat{e}:\mathcal{A}\times D \rightarrow \mathbb{R}_{\ge 0}$ of the operators $\mathcal{G}^\dagger$, $\hat{\mathcal{G}}$ and $\hat{\mathcal{E}}$. This assumption holds whenever $\mathcal{U}$ embeds continuously in $C(\overline{D};\mathbb{R}^{d_u})$ or in some $L^p(D;\mathbb{R}^{d_u})$ for $1 \le p < \infty$, which covers every Lebesgue, Sobolev and H\"older space on $D$. For $\hat{\mathcal{G}}$ and $\hat{\mathcal{E}}$ this is immediate: FNOs are continuous operators with trigonometric polynomial outputs, so the continuous representative is unique and depends measurably on $a$. For $\mathcal{G}^\dagger$ only Borel measurability of the operator is required, which follows from Hadamard well-posedness of the PDE.

\begin{definition}[Prediction band] \label{def:prediction band}
  Let $\lambda \in [0, \infty]$, $a \in \mathcal{A}$, and $x \in D$, we define the prediction band $C_\lambda(a)(x) \subset \mathbb{R}^{d_u}$ around the prediction $\hat{\mathcal{G}}(a)(x)$ as
  \[
    C_\lambda(a)(x) := \bigl\{ y \in \mathbb{R}^{d_u} :
    \|y - \hat{\mathcal{G}}(a)(x)\|_2 \le \lambda\, \hat{\mathcal{E}}(a)(x) \bigr\}= \bar{B}(\hat{\mathcal{G}}(a)(x), \lambda \hat{\mathcal{E}}(a)(x)),
  \]
  with the convention that $C_\infty(a)(x):=\mathbb{R}^{d_u}$.
\end{definition}

The constant $\lambda\in [0, \infty]$ is a ``scaling factor'', which uniformly scales the error estimator $\hat{\mathcal{E}}(a)(x)$. \cref{lemma:rep-indep} shows that the containment fraction of the domain, i.e., $\nu(\{x \in D : \mathcal{G}^\dagger(a)(x) \in C_\lambda(a)(x)\})$,  is independent of the choice of representatives for the operators, so that one can drop the representative notation and write $\mathcal{G}^\dagger$, $\hat{\mathcal{G}}$, $\hat{\mathcal{E}}$ for any fixed choice of product-measurable representatives.

Let $\mu$ be a probability measure on $(\mathcal{A},\mathcal{B}(\mathcal{A}))$, and fix a \emph{spatial tolerance} $\gamma\in(0,1)$, the fraction of the domain on which containment may fail for a given input, and a \emph{probability tolerance} $\alpha\in(0,1)$, the maximum probability that this spatial requirement may fail. Let $A_1,\dots,A_{n+1}$ be exchangeable input functions with common marginal law $\mu$, and denote the calibration set by $\mathcal{D}_{\mathrm{cal}}=(A_1,\dots,A_n)$. We aim to choose a scaling factor $\hat\lambda=\hat\lambda(\mathcal{D}_{\mathrm{cal}})$ such that
\begin{equation}\label{eq:marginal guarantee}
  \mathbb{P}_{A_1,\ldots,A_{n+1}}\bigl( \nu \bigl(\{ x \in D :
  \mathcal{G}^\dagger(A_{n+1})( x) \in C_{\hat{\lambda}}(A_{n+1})(x) \}\bigr)
  \ge 1 - \gamma \bigr) \ge 1 - \alpha,
\end{equation}
i.e., the true output for the test input is contained on at least a fraction $1-\gamma$ of $D$ with probability at least $1-\alpha$. This is a marginal guarantee over the joint draw of the calibration set and test input, and exchangeability alone is sufficient. Under the assumption that $A_1,\dots,A_{n+1}$ are independent and identically distributed (i.i.d.), the calibration set is independent of the test input and the marginal probability in \cref{eq:marginal guarantee} decomposes as
\begin{equation}\label{eq:conditional decomposition}
  \mathbb{E}_{\mathcal{D}_{\mathrm{cal}}}\Bigl[\,
    \mathbb{P}_{A_{n+1}}\bigl( \nu(\{ x : \mathcal{G}^\dagger(A_{n+1})(x) \in
    C_{\hat\lambda}(A_{n+1})(x)\}) \ge 1-\gamma \,\bigm|\,
    \mathcal{D}_{\mathrm{cal}} \bigr) \Bigr] ,
\end{equation}
an average over calibration sets of the conditional probability of sufficient containment. The marginal guarantee controls this average, not the conditional probability for every realized calibration set.

\paragraph{Coverage for arbitrary residual fields.}
Before specializing to neural operators, \cref{sec:general proof} establishes the conformal argument in a general setting: given a measurable residual field $r:\mathcal{A}\times X\to[0,\infty]$ on a probability space $(X,\mathcal{X},\pi)$, we define a scalar score as the smallest threshold that contains at least a $1-\gamma$ fraction of the evaluation space $X$ and calibrate an order statistic of these scores to obtain marginal coverage at least $1-\alpha$. Here, the argument only depends on measurability of the field and nothing else, not the model, the PDE, or even that $X$ be a subset of the domain $D$. The sections that follow apply \cref{thm:general guarantee} by specifying the evaluation space, its probability measure, and the relevant residual fields: first for the idealized continuum problem, then for discrete observations.

\subsection{Continuum coverage guarantee}\label{sec:continuum guarantee}

We first consider the idealized setting in which the truth, prediction, and error estimator can be evaluated at every point $x\in D$, instead of at a finite set of points. To define the scaling factor $\lambda$ associated with the prediction band $C_{\lambda}$ from \cref{def:prediction band}, we introduce the continuum residual field $\rho:\mathcal{A}\times D\to[0,\infty]$ as pointwise prediction error normalized by the error estimator:
\begin{equation}\label{eq:continuum residual}
  \rho(a,x)=\frac{\|\mathcal{G}^\dagger(a)(x)-\hat{\mathcal{G}}(a)(x)\|_2}{\hat{\mathcal{E}}(a)(x)}
  \in[0,\infty],\quad a\in\mathcal{A},\,x\in D,
\end{equation}
where $\rho(a,x)=0$ if $\hat{\mathcal{E}}(a)(x)=0$ and $\mathcal{G}^\dagger(a)(x)=\hat{\mathcal{G}}(a)(x)$, and $\rho(a,x)=\infty$ if $\hat{\mathcal{E}}(a)(x)=0$ and $\mathcal{G}^\dagger(a)(x)\neq\hat{\mathcal{G}}(a)(x)$. Since the operators $\mathcal{G}^\dagger$, $\hat{\mathcal{G}}$, and $\hat{\mathcal{E}}$ are Borel measurable, the residual field $\rho$ is jointly measurable with respect to the product sigma algebra $\mathcal{B}(\mathcal{A})\otimes\mathcal{B}(D)$. Then, for every $a\in\mathcal{A}$ and $x\in D$, $\mathcal{G}^\dagger(a)(x)\in C_\lambda(a)(x)$ if and only if $\rho(a,x)\leq\lambda$, that is $\{x\in D:\rho(a,x)\leq\lambda\} = \{x\in D:\mathcal{G}^\dagger(a)(x)\in C_\lambda(a)(x)\}$ so that one can apply the results from \cref{sec:general proof} to obtain a marginal coverage guarantee for the continuum problem.

\paragraph{Choice of scaling factor.}
Given $a\in\mathcal{A}$, we respectively define the ``containment fraction'' $w_\lambda(a)$ of the domain for which the residual is smaller than $\lambda$ and the corresponding scalar nonconformity score $s(a)$, measuring the spatial $(1-\gamma)$-quantile, as
\begin{equation}\label{eq:continuum score}
  w_\lambda(a)=\nu\bigl(\{x\in D:\rho(a,x)\leq\lambda\}\bigr),
  \quad \text{and}\quad
  s(a)=\inf\{\lambda\geq0:w_\lambda(a)\geq1-\gamma\}.
\end{equation}
Let $k = \lceil(1-\alpha)(n+1)\rceil$, and consider calibration inputs $A_1,\dots,A_n\in\mathcal{A}$. The scaling factor $\hat\lambda$ is chosen as the $k$th smallest value of the nonconformity scores $s(A_1),\dots,s(A_n)$ if $k\leq n$ and $\infty$ otherwise. \cref{lemma:event-meas} then ensures that the nonconformity score $s$ and the resulting calibration map $\hat{\lambda}$ (as a function of the calibration inputs $A_1,\dots,A_n$) are measurable.

\begin{remark}[Ideal error estimator]\label{remark:exact estimator}
  If $\hat{\mathcal{E}}(a)(x)=\|\mathcal{G}^\dagger(a)(x)-\hat{\mathcal{G}}(a)(x)\|_2$ for all $x\in D$, then $\rho(a,x)=1$ for every $a$. Every calibration score then satisfies $s(A_i)=1$, so the calibrated scaling factor is $\hat\lambda=1$, and the prediction band $C_1(a)(x)$ is a ball of radius $\hat{\mathcal{E}}(a)(x)$ whose boundary passes exactly through $\mathcal{G}^\dagger(a)(x)$, achieving coverage over the entire domain $D$.
\end{remark}

The following result is the continuum instantiation of the split-conformal guarantee proved in \cref{thm:general guarantee}, with $(X,\mathcal{X},\pi)=(D,\mathcal{B}(D),\nu)$ and $r=\rho$.

\begin{theorem}[Continuum marginal coverage]\label{thm:continuum guarantee}
  Suppose that $A_1,\dots,A_{n+1}\sim \mu$ are exchangeable, then
  \[
    \mathbb{P}_{A_1,\dots,A_{n+1}}\!\left(
    \nu\bigl(\{x\in D:\mathcal{G}^\dagger(A_{n+1})(x)
    \in C_{\hat\lambda}(A_{n+1})(x)\}\bigr)
    \geq1-\gamma
    \right)
    \geq\frac{k}{n+1}\geq1-\alpha.
  \]
  Moreover, if the $A_i$ are i.i.d.\ and the distribution of $s(A_1)$ is atomless, then the random variable $\mathbb{P}\bigl( w_{\hat\lambda}(A_{n+1}) \ge 1-\gamma \,\bigm|\,
    A_1,\dots,A_n \bigr)$ follows a $\mathrm{Beta}(k,\, n+1-k)$ distribution and
  \[
    \mathbb{P}_{\mu^{\otimes (n+1)}}\!\left(
    \nu\bigl(\{x\in D:\mathcal{G}^\dagger(A_{n+1})(x)
    \in C_{\hat\lambda}(A_{n+1})(x)\}\bigr)
    \geq1-\gamma
    \right)
    =\frac{k}{n+1}\leq 1-\alpha+\frac{1}{n+1}.
  \]
\end{theorem}

\begin{remark}[Finiteness of the scaling factor]\label{remark:cont finite scaling}
  If the error estimator is constructed (e.g. with a strictly positive final activation function like softplus~\citealt{dugas2000incorporating}) such that $\hat{\mathcal{E}}(a)(x) > 0$ $\nu$-a.e.\ for $\mu$-a.e.\ $a\in\mathcal{A}$, then \cref{lemma:finite score} implies that the scaling factor satisfies $\hat\lambda < \infty$ $\mu^{\otimes n}$-a.s.\ provided $\lceil(1-\alpha)(n+1)\rceil \le n$.
\end{remark}

The guarantee in \cref{thm:continuum guarantee} is a marginal coverage guarantee that is finite-sample and distribution-free. While it requires exchangeability of the calibration sets, it does not require any accuracy assumptions on $\hat{\mathcal{G}}$ or $\hat{\mathcal{E}}$. The second part of \cref{thm:continuum guarantee} provides a precise characterization of the marginal coverage probability in the i.i.d.\ case, which is useful for understanding the tightness of the guarantee. This result provides a way to quantify the uncertainty associated with the coverage guarantee and to assess the reliability of the prediction bands generated by the conformal method. The atomless assumption ensures that the nonconformity scores are continuous, which is common in operator learning where the input functions are often drawn from a Gaussian process~\citep{kovachki2023neural,boulle2024mathematical}. Finally, we remark that the continuum score in \cref{eq:continuum score} may not be directly computable from finitely many function evaluations, which motivates the discrete guarantees developed in the next section.

\subsection{Discrete coverage guarantee}\label{subsec:grid guarantee}

We now consider the practical setting in which the solutions, predictions, and the error estimations are evaluated at a discrete set of points $D_M=\{x_1,\dots,x_M\}\subset D$. Rather than measuring containment over the continuum domain, our goal is to guarantee that the solution is contained at a sufficiently large fraction of these grid points. We therefore equip $D_M$ with the power-set sigma algebra $2^{D_M}$ and the uniform counting probability measure $
  \nu_M(B)
  =\tfrac{1}{M}\sum_{j=1}^{M}\delta_{x_j}(B)$ for $B\subseteq D_M$.
Let $\hat{\mathcal{G}}_M$ and $\hat{\mathcal{E}}_M$ denote the prediction and error estimators evaluated on $D_M$. For $\lambda\in[0,\infty]$, we define the discrete prediction band by
$
  C^M_\lambda(a)(x_j)
  =\bigl\{y\in\mathbb{R}^{d_u}:
  \|y-\hat{\mathcal{G}}_M(a)(x_j)\|_2
  \leq\lambda\hat{\mathcal{E}}_M(a)(x_j)\bigr\}$
and introduce the discrete residual field $\rho_M:\mathcal{A}\times D_M\to[0,\infty)$ as the pointwise error between the prediction and solution, normalized by the error estimator:
\begin{equation}\label{eq:grid residual}
  \rho_M(a,x_j)
  =\frac{\|\mathcal{G}^\dagger(a)(x_j)-\hat{\mathcal{G}}_M(a)(x_j)\|_2}
  {\hat{\mathcal{E}}_M(a)(x_j)},\quad a\in\mathcal{A},\,x_j\in D_M.
\end{equation}
Here, we assume that $a\mapsto \mathcal{G}^\dagger(a)(x_j)$ and the coordinate maps of $\hat{\mathcal{G}}_M$ and $\hat{\mathcal{E}}_M$ are Borel measurable for every $j$, and that $\hat{\mathcal{E}}_M(a)(x_j)>0$. Since $D_M$ is finite, $\rho_M$ is jointly measurable with respect to $\mathcal{B}(\mathcal{A})\otimes2^{D_M}$. Then, for every $a\in\mathcal{A}$ and $x_j\in D_M$, $\mathcal{G}^\dagger(a)(x_j)\in C^M_\lambda(a)(x_j)$ if and only if $\rho_M(a,x_j)\leq\lambda$, so the results of \cref{sec:general proof} also apply to the discrete problem.

We define the fraction of grid points for which the residual is smaller than $\lambda$ and the corresponding scalar nonconformity score similarly to \cref{sec:continuum guarantee} as $w^M_\lambda(a)=\frac{1}{M}\sum_{j=1}^{M}\mathbf{1}\{\rho_M(a,x_j)\leq\lambda\}$ and
$s_M(a)=\inf\{\lambda\geq0:w^M_\lambda(a)\geq1-\gamma\}$.
Unlike the continuum case, this spatial quantile can be computed exactly from finitely many evaluations: if $q=\lceil(1-\gamma)M\rceil$, then $s_M(a)$ is the $q$th smallest value among $\rho_M(a,x_1),\dots,\rho_M(a,x_M)$. Let $k=\lceil(1-\alpha)(n+1)\rceil$, given calibration inputs $A_1,\dots,A_n$, the scaling factor $\hat\lambda_M$ is chosen as the $k$th smallest value of the scores $s_M(A_1),\dots,s_M(A_n)$ if $k\leq n$ and $\infty$ otherwise. The following result is a corollary of the generic split-conformal guarantee proved in \cref{thm:general guarantee} with $(X,\mathcal{X},\pi)=(D_M,2^{D_M},\nu_M)$ and $r=\rho_M$.

\begin{theorem}[Discrete marginal coverage]\label{thm:grid guarantee}
  Suppose that $A_1,\dots,A_{n+1}\sim\mu$ are exchangeable, then
  \[
    \mathbb{P}\!\left(
    \nu_M\bigl(\{x_j\in D_M:\mathcal{G}^\dagger(A_{n+1})(x_j)
    \in C^M_{\hat\lambda_M}(A_{n+1})(x_j)\}\bigr)
    \geq1-\gamma
    \right)
    \geq\frac{k}{n+1}\geq1-\alpha.
  \]
  Moreover, if the $A_i$ are i.i.d.\ and the distribution of $s_M(A_1)$ is atomless, then the random variable
  $\mathbb{P}(w^M_{\hat\lambda_M}(A_{n+1})\geq1-\gamma\mid A_1,\dots,A_n)$ follows a $\mathrm{Beta}(k,n+1-k)$ distribution and
  \[
    \mathbb{P}\!\left(
    \nu_M\bigl(\{x_j\in D_M:\mathcal{G}^\dagger(A_{n+1})(x_j)
    \in C^M_{\hat\lambda_M}(A_{n+1})(x_j)\}\bigr)
    \geq1-\gamma
    \right)
    =\frac{k}{n+1}\leq1-\alpha+\frac{1}{n+1}.
  \]
\end{theorem}

\begin{remark}[Calibrating on the discrete score]\label{remark:dominating score continuum}
  While the continuum score $s$ from \cref{sec:continuum guarantee} may not be computable, \cref{cor:dominating score} shows that calibrating on any Borel score $\sigma:\mathcal{A}\to[0,\infty]$ satisfying $s(a)\leq\sigma(a)$ $\mu$-a.s.\ still yields marginal coverage at level $1-\alpha$ for the continuum residual field $\rho$. In particular, this opens the door to constructing computable scores that dominate $s$, thereby providing the continuum guarantee of \cref{thm:continuum guarantee} from a discrete implementation.
\end{remark}

\section{Function-Valued Split Conformal Calibration}\label{sec:methodology}

We now translate the coverage guarantees from \cref{sec:theory} into a concrete calibration procedure. Let $\mathcal{D}=\{(a_i,u_i)\}_{i=1}^{N_{\text{sample}}}\subset \mathcal{A}\times \mathcal{U}$ be a labelled dataset such that $u_i=\mathcal{G}^\dagger(a_i)$ for $1\leq i\leq N_{\text{sample}}$. We randomly split $\mathcal{D}$ into three disjoint training, residual, and calibration sets. In practice, we typically choose around $60\%$ of the data for training, $30\%$ for residual, and $10\%$ for calibration, but this depends on the difficulty of the problem. The training and residual datasets are used to train the prediction operator and error estimator $\hat{\mathcal{G}}$ and $\hat{\mathcal{E}}$, while the calibration set is held out until both operators are frozen. This separation is essential as the finite-sample guarantees from \cref{sec:theory} treat the operators as fixed and require the calibration and test inputs to be exchangeable. Although splitting reduces the data available to fit each model, operator learning has favorable sample complexity in several settings, including elliptic PDE solution operators and noisy linear operators \citep{boulle2023elliptic,dehoop2023convergence}. These results suggest that moderate holdout sets need not substantially degrade predictive accuracy, although the effect remains problem-dependent.

\begin{algorithm}[htbp]
  \caption{Function-valued split conformal calibration}\label{alg:functional conformal}
  \begin{algorithmic}[1]
    \algrenewcommand\algorithmicrequire{\textbf{Input:}}
    \Require Labelled data $\mathcal{D}=\{(a_i,u_i)\}_{i=1}^{N_{\text{sample}}}$, spatial tolerance $\gamma\in(0,1)$, and probability tolerance $\alpha\in(0,1)$
    \Statex \hspace{-\labelwidth}\hspace{-\labelsep}\textbf{Data splitting and model fitting}
    \State Split $\mathcal{D}$ into three disjoint datasets: $\mathcal{D}_{\mathrm{train}}$, $\mathcal{D}_{\mathrm{res}}$, and $\mathcal{D}_{\mathrm{cal}}=\{(a_i,u_i)\}_{i=1}^n$
    \State Train the prediction operator $\hat{\mathcal{G}}$ to approximate $u_i(x) = \mathcal{G}^\dagger(a_i)(x)$ on $\mathcal{D}_{\mathrm{train}}$ and freeze it
    \State Train the error estimator $\hat{\mathcal{E}}$ to approximate $\|u_i(x)-\hat{\mathcal{G}}(a_i)(x)\|_2$ on $\mathcal{D}_{\mathrm{res}}$ and freeze it
    \Statex \hspace{-\labelwidth}\hspace{-\labelsep}\textbf{Conformal calibration}
    \For{$i=1,\dots,n$}
    \State Compute the residual field $\displaystyle \rho_i(x)=
      \|u_i(x)-\hat{\mathcal{G}}(a_i)(x)\|_2/
      \hat{\mathcal{E}}(a_i)(x)$ for $x\in D$
    \State Set $\displaystyle S_i=\inf\bigl\{\lambda\geq0:
      \nu(\{x\in D:\rho_i(x)\leq\lambda\})\geq1-\gamma\bigr\}$ \label{line:score}
    \EndFor
    \State Set $k=\lceil(1-\alpha)(n+1)\rceil$ and select $\hat\lambda$ as the $k$th smallest value among $S_1,\dots,S_n$ if $k\leq n$, and $\hat\lambda=\infty$ otherwise
    \State \Return Scaling factor $\hat\lambda$ and prediction band operator $C_{\hat\lambda}: (a,x)\mapsto \bar{B}(\hat{\mathcal{G}}(a)(x), \hat\lambda \hat{\mathcal{E}}(a)(x))$
  \end{algorithmic}
\end{algorithm}

Once the prediction and error estimators are trained, the calibration procedure in \cref{alg:functional conformal} computes the residual field $\rho_i$ for each calibration input $a_i\in \mathcal{D}_{\text{cal}}$. The scalar score $S_i$ is defined as the smallest threshold $\lambda$ such that the residual field $\rho_i$ does not exceed $\lambda$ on at least a fraction $1-\gamma$ of the domain, i.e.\ the spatial $(1-\gamma)$-quantile of $\rho_i$ under $\nu$. We then select the scaling factor $\hat\lambda$ as the $k$th order statistic of these scores, where $k=\lceil(1-\alpha)(n+1)\rceil$. This choice ensures that the resulting prediction band $C_{\hat\lambda}(A_{n+1})$ satisfies the marginal coverage guarantee in \cref{sec:theory}. While the algorithm is stated for the continuum setting, it applies equally to the discrete setting by replacing $D$ with the discrete set $D_M$ and $\nu$ with the uniform counting measure $\nu_M$. In this case, the spatial infimum in line~\ref{line:score} reduces to the $q$th order statistic of the $M$ residual values, with $q=\lceil(1-\gamma)M\rceil$, and is computed by sorting (see \cref{subsec:grid guarantee}).
Given the calibration set, the conditional coverage, i.e., the probability that a test input $A_{n+1}$ is covered, follows a Beta distribution:
\[
  Q := \mathbb{P}_{A_{n+1}}\!\left(
  \nu\bigl(\{x\in D:\mathcal{G}^\dagger(A_{n+1})(x)\in C_{\hat\lambda}(A_{n+1})(x)\}\bigr)\geq1-\gamma
  \mid \mathcal{D}_{\text{cal}}\right) \sim\operatorname{Beta}(k,n+1-k),
\]
under the i.i.d.\ and atomless-score conditions of \cref{thm:continuum guarantee,thm:grid guarantee}. This characterization quantifies the variability of the conditional coverage across calibration sets since
$\operatorname{Var}(Q)
  =\tfrac{k(n+1-k)}{(n+1)^2(n+2)}\underset{n\to\infty}{\sim}\tfrac{\alpha(1-\alpha)}{n}$.
Hence, the standard deviation of the conditional coverage is $O(n^{-1/2})$, which decreases as the calibration-set size increases.

\section{Numerical Experiments}\label{sec:experiments}

We compare our functional conformal calibration (\cref{alg:functional conformal}) with the method of \citet{ma2024calibrated}. Both methods use the same retrained operators $(\hat{\mathcal{G}}_M,\hat{\mathcal{E}}_M)$ and calibration data to obtain a fair comparison, and produce pointwise bands $\hat{\mathcal{G}}_M(a)(x_j) \pm \hat{\lambda}\hat{\mathcal{E}}_M(a)(x_j)$ for each input $a$ and grid node $x_j$, since in this setting $d_u=1$. The neural operator predictor $\hat{\mathcal{G}}_M$ and error estimator $\hat{\mathcal{E}}_M$ are Fourier neural operators (FNOs) of identical architecture, with a final Softplus activation in $\hat{\mathcal{E}}_M$ to enforce non-negativity of the output. We perform numerical experiments on two standard benchmarks from \citep{li2021fourier}: a steady 2D Darcy flow (also used in \citealt{ma2024calibrated}) and a more challenging 2D incompressible Navier--Stokes dataset in vorticity form at $\nu=10^{-4}$. Each dataset is split once with a fixed seed into four disjoint subsets as training, residual, calibration, and test sets following \cref{sec:methodology}. We train $\hat{\mathcal{G}}_M$ with relative $L^2$ loss \citep{kossaifi2025librarylearningneuraloperators} and AdamW, then freeze it. We then train $\hat{\mathcal{E}}_M$ on pointwise residuals $\|\mathcal{G}^\dagger(a)(x_j)-\hat{\mathcal{G}}_M(a)(x_j)\|_2$ with pinball loss at level $1-\gamma$, so that it estimates the pointwise $(1-\gamma)$-quantile of the prediction error. Additional implementation details are available in \cref{app:implem details}.

Given a dataset on a two-dimensional grid at resolution $N_{\max}$, we subsample it to obtain coarser grids of resolution $N$, with $M=N^2$ number of points. At each resolution, we train a new $(\hat{\mathcal{G}}_M,\hat{\mathcal{E}}_M)$ pair and compare two calibration rules: ours (\cref{alg:functional conformal}) and \citep{ma2024calibrated}, using the same trained operators and calibration data. Throughout, containment fraction denotes the fraction of grid points for which a given input is contained in the prediction band, and coverage is the fraction of test inputs whose containment fraction is greater than $1-\gamma$. In the experiments that follow, we select $\gamma=\alpha=0.1$; we are interested in a $90\%$ containment of the grid points, over $90\%$ of calibration/test draws. We then perform $3000$ random disjoint re-splits of the calibration and test functions, recalibrate each time, and report the empirical distribution of $\hat{\lambda}$. The validity of the method is assessed using the same resampling procedure, where we compute the empirical coverage on each re-split. Finally, we evaluate the transfer of a pair calibrated at $N<N_{\max}$ to $N_{\max}$ by either spectral interpolation of coarse outputs to fine nodes, or zero-shot super-resolution.

\paragraph{Two-dimensional Darcy flow.}

In \cref{fig:darcy-main}(left), we report the median calibrated scaling factors along with the $5$th--$95$th percentiles over the $3000$ random disjoint re-splits of the calibration and test sets, with the trained operators held fixed. The calibrated factors given by our method are consistently smaller than those of \citep{ma2024calibrated}, across spatial resolutions with ratio ranging from $2.37\times$ at $N=12$ to $1.16\times$ at $N=84$ (\cref{tab:darcy-efficiency}). While the correction of \citep{ma2024calibrated} shrinks with resolution (roughly at a scale of $\sqrt{-\log\alpha/(2N^d)}$), it remains strictly conservative at practical resolutions. \cref{fig:darcy-main}(right) shows that our tighter scaling factors do not compromise validity: the empirical distribution of resampled coverage at $N=60$ closely matches the Beta-Binomial law ($\mathrm{BetaBinomial}\bigl(n_{\mathrm{test}},k,n_{\mathrm{cal}}+1-k\bigr)$ with
$k=\lceil(1-\alpha)(n_{\mathrm{cal}}+1)\rceil$) predicted by \cref{thm:grid guarantee}, indicating near-exact finite-sample calibration on this grid. In contrast, the empirical coverage obtained by \citet{ma2024calibrated} concentrates near $0.98$, which ensures validity but is substantially more conservative than the target $1-\alpha=0.9$.

\begin{figure}[htbp]
  \hspace*{-1cm}
  \centering
  \begin{overpic}[width=0.9\textwidth]{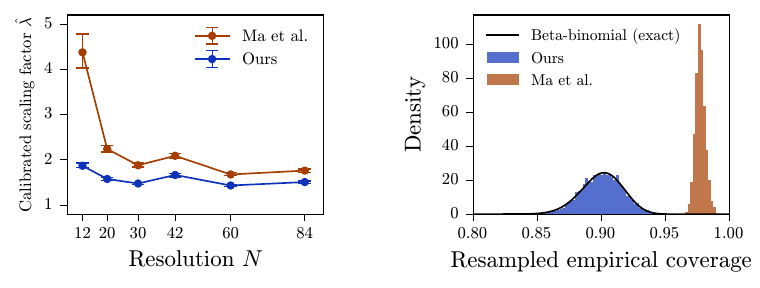}
    \put(76.3,17){\small$1-\alpha$}
  \end{overpic}
  \caption{2D Darcy flow dataset. Left: median calibrated scaling factor $\hat{\lambda}$ with respect to the spatial resolution $N$ over $3000$ random disjoint re-splits ($n_{\mathrm{cal}}=500$, $n_{\mathrm{test}}=1000$) along with empirical $5$th--$95$th percentiles; smaller $\hat{\lambda}$ indicates tighter bands. Right: empirical distribution of resampled coverage at resolution $N=60$ compared with the exact Beta-Binomial law predicted by \cref{thm:grid guarantee} at the target level $1-\alpha = 0.9$. Our distribution aligns with the predicted law, while \citep{ma2024calibrated} is shifted right with more conservative and wider bands.}
  \label{fig:darcy-main}
\end{figure}

The violin plots in \cref{fig:darcy-violins} display the same behavior across all spatial resolutions $N$ considered, with complete conservativeness (point mass near $1$) at coarse resolutions. Finally, \cref{fig:darcy-transfer} in \cref{app:additional numerical} evaluates transfer of the coverage guarantees to the fine grid. Under zero-shot super-resolution, we note that the coverage is zero at all calibration resolutions for both methods, while under spectral interpolation, coverage is near zero for $N\leq 42$ and increases with the spatial resolution. While the correction of \citep{ma2024calibrated} is designed to provide exactly this transfer by inflating the scaling factor, it only achieves the desired coverage at resolution $N = 84$. In this case, the discretization is fine enough so that no correction is needed and our method achieves the target coverage without any inflation.

\begin{figure}[t]
  \centering
  \includegraphics[width=0.9\textwidth]{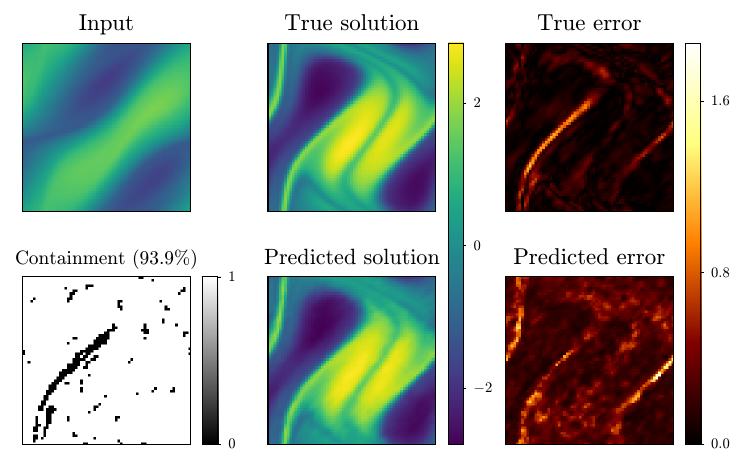}
  \caption{Navier--Stokes at $N=64$: input vorticity, solution at $t=21$, FNO prediction, pointwise error, and estimated error field for a calibrated operator pair. The estimated error tracks the spatial error pattern, enabling tight pointwise bands while maintaining the target containment fraction $1-\gamma$.}
  \label{fig:ns-vis}
\end{figure}

\paragraph{2D incompressible Navier-Stokes.}
Given an initial vorticity field, we aim to predict the solution at a later time $t=21$. The problem is more challenging than Darcy flow, with a higher prediction error (about $13\%$ relative error) and a more complex spatial structure to capture. We train the same FNO architectures as in Darcy flow, but with a larger number of Fourier modes to capture the finer details of the solution fields. We display in \cref{fig:ns-vis}, a visualization of the predicted solution for a test input, along with the pointwise error and the scaled (our calibration method) estimated error field. Additional evaluations in \cref{app:additional numerical} show that the empirical coverage distribution again follows the Beta-Binomial law predicted by \cref{thm:grid guarantee} (cf. \cref{fig:ns-main,fig:ns-violins}), and that our method consistently produces tighter bands than \citet{ma2024calibrated} across all resolutions considered. As observed in \cref{tab:ns-efficiency}, the ratio of the scaling factors produced by their method \citep{ma2024calibrated} relative to ours ranges from $1.78\times$ at $N=16$ to $1.23\times$ at $N=64$, while preserving the target coverage. The empirical coverage distribution for \citet{ma2024calibrated} is again shifted to higher values, confirming its conservative nature.

\section{Conclusion}
In this paper, we developed a split conformal framework to provide theoretically justified uncertainty quantification guarantees for neural operators. The calibration procedure constructs prediction bands based on a held-out calibration sets that are guaranteed to contain the true solution with a tolerance probability, while allowing for a controlled fraction of spatial failure over the domain. The construction reduces a measurable residual field to a spatial quantile and calibrates a single scaling factor, yielding finite-sample marginal coverage under exchangeability. The proof strategy applies to both continuum domains and fixed grids, with a discrete guarantee that does not rely on random spatial sampling or concentration corrections, and we characterized calibration-set variability via an exact Beta law under i.i.d.\ sampling and atomless scores. Empirically, on Darcy flow and Navier--Stokes equations, coverage matches the predicted behavior while achieving systematically tighter bands than \citep{ma2024calibrated}, whose corrections are more conservative. The approach is model-agnostic and applies to any neural operator architecture, provided that an appropriate error estimator is available.

\paragraph{Limitations.}
Our guarantees are marginal over the calibration and test inputs and assume their exchangeability; they do not ensure coverage under distribution shift or guarantee the target coverage conditional on every realized calibration set. Finally, the discrete guarantee is tied to the calibration grid and is not automatically invariant to the spatial resolution. \Cref{cor:dominating score} offers a framework for lifting it to the continuum, via a computable score dominating the continuum score. Constructing such a score requires regularity of the solution operator and the learned fields, is problem-dependent, and may produce conservative bands.

\bibliographystyle{plainnat_jabbrv}
\bibliography{refs}

\begin{thebibliography}{51}
\providecommand{\natexlab}[1]{#1}
\providecommand{\url}[1]{\texttt{#1}}
\expandafter\ifx\csname urlstyle\endcsname\relax
  \providecommand{\doi}[1]{doi: #1}\else
  \providecommand{\doi}{doi: \begingroup \urlstyle{rm}\Url}\fi

\bibitem[Angelopoulos and Bates(2023)]{angelopoulos2023gentle}
Anastasios~N Angelopoulos and Stephen Bates.
\newblock Conformal prediction: {A} gentle introduction.
\newblock {\em\protect\JournalTitle{Foundations and Trends in Machine
  Learning}}, 16\penalty0 (4):\penalty0 494--591, 2023.

\bibitem[Angelopoulos et~al.(2024)Angelopoulos, Barber, and
  Bates]{angelopoulos2024conformalbook}
Anastasios~N Angelopoulos, Rina~Foygel Barber, and Stephen Bates.
\newblock Theoretical foundations of conformal prediction.
\newblock {\em\protect\JournalTitle{arXiv preprint arXiv:2411.11824}}, 2024.

\bibitem[Bi et~al.(2023)Bi, Xie, Zhang, Chen, Gu, and Tian]{bi2023pangu}
Kaifeng Bi, Lingxi Xie, Hengheng Zhang, Xin Chen, Xiaotao Gu, and Qi~Tian.
\newblock Accurate medium-range global weather forecasting with {3D} neural
  networks.
\newblock {\em\protect\JournalTitle{Nature}}, 619\penalty0 (7970):\penalty0
  533--538, 2023.

\bibitem[Boull{\'e} and Townsend(2024)]{boulle2024mathematical}
Nicolas Boull{\'e} and Alex Townsend.
\newblock A mathematical guide to operator learning.
\newblock In {\em\protect\JournalTitle{Handbook of Numerical Analysis}},
  volume~25, pages 83--125. Elsevier, 2024.

\bibitem[Boull{\'e} et~al.(2023)Boull{\'e}, Halikias, and
  Townsend]{boulle2023elliptic}
Nicolas Boull{\'e}, Diana Halikias, and Alex Townsend.
\newblock Elliptic {PDE} learning is provably data-efficient.
\newblock {\em\protect\JournalTitle{Proceedings of the National Academy of
  Sciences USA}}, 120\penalty0 (39):\penalty0 e2303904120, 2023.

\bibitem[Boull{\'e} et~al.(2024)Boull{\'e}, Halikias, Otto, and
  Townsend]{boulle2024operator}
Nicolas Boull{\'e}, Diana Halikias, Samuel~E Otto, and Alex Townsend.
\newblock Operator learning without the adjoint.
\newblock {\em\protect\JournalTitle{Journal of Machine Learning Research}},
  25\penalty0 (364):\penalty0 1--54, 2024.

\bibitem[Bouziani and Boull{\'e}(2024)]{bouziani2024structure}
Nacime Bouziani and Nicolas Boull{\'e}.
\newblock Structure-preserving operator learning.
\newblock {\em\protect\JournalTitle{arXiv preprint arXiv:2410.01065}}, 2024.

\bibitem[B{\"u}lte et~al.(2025)B{\"u}lte, Scholl, and
  Kutyniok]{bulte2025probabilistic}
Christopher B{\"u}lte, Philipp Scholl, and Gitta Kutyniok.
\newblock Probabilistic neural operators for functional uncertainty
  quantification.
\newblock {\em\protect\JournalTitle{Transactions on Machine Learning
  Research}}, 2025.

\bibitem[Cao(2021)]{cao2021choose}
Shuhao Cao.
\newblock Choose a transformer: {Fourier} or {Galerkin}.
\newblock In {\em\protect\JournalTitle{Advances in Neural Information
  Processing Systems}}, volume~34, 2021.

\bibitem[de~Hoop et~al.(2023)de~Hoop, Kovachki, Nelsen, and
  Stuart]{dehoop2023convergence}
Maarten~V de~Hoop, Nikola~B Kovachki, Nicholas~H Nelsen, and Andrew~M Stuart.
\newblock Convergence rates for learning linear operators from noisy data.
\newblock {\em\protect\JournalTitle{SIAM/ASA Journal on Uncertainty
  Quantification}}, 11\penalty0 (2):\penalty0 480--513, 2023.

\bibitem[Dugas et~al.(2000)Dugas, Bengio, B{\'e}lisle, Nadeau, and
  Garcia]{dugas2000incorporating}
Charles Dugas, Yoshua Bengio, Fran{\c{c}}ois B{\'e}lisle, Claude Nadeau, and
  Ren{\'e} Garcia.
\newblock Incorporating second-order functional knowledge for better option
  pricing.
\newblock In {\em\protect\JournalTitle{Advances in Neural Information
  Processing Systems}}, volume~13, 2000.

\bibitem[Durrett(2019)]{durrett2019probability}
Rick Durrett.
\newblock \emph{Probability: Theory and Examples}.
\newblock Cambridge University Press, 5th edition, 2019.

\bibitem[Gal and Ghahramani(2016)]{gal2016dropout}
Yarin Gal and Zoubin Ghahramani.
\newblock {Dropout as a bayesian approximation: Representing model uncertainty
  in deep learning}.
\newblock In {\em\protect\JournalTitle{International Conference on Machine
  Learning}}, pages 1050--1059, 2016.

\bibitem[Gawlikowski et~al.(2023)Gawlikowski, Tassi, Ali, Lee, Humt, Feng,
  Kruspe, Triebel, Jung, Roscher, Shahzad, Yang, Bamler, and
  Zhu]{gawlikowski2023survey}
Jakob Gawlikowski, Cedrique Rovile~Njieutcheu Tassi, Mohsin Ali, Jongseok Lee,
  Matthias Humt, Jianxiang Feng, Anna Kruspe, Rudolph Triebel, Peter Jung,
  Ribana Roscher, Muhammad Shahzad, Wen Yang, Richard Bamler, and Xiao~Xiang
  Zhu.
\newblock A survey of uncertainty in deep neural networks.
\newblock {\em\protect\JournalTitle{Artif. Intell. Rev.}}, 56:\penalty0
  1513--1589, 2023.

\bibitem[Gopakumar et~al.(2024)Gopakumar, Oskarrson, Gray, Zanisi, Pamela,
  Giles, Kusner, and Deisenroth]{gopakumar2024valid}
Vignesh Gopakumar, Joel Oskarrson, Ander Gray, Lorenzo Zanisi, Stanislas
  Pamela, Daniel Giles, Matt Kusner, and Marc Deisenroth.
\newblock Valid error bars for neural weather models using conformal
  prediction.
\newblock {\em\protect\JournalTitle{arXiv preprint arXiv:2406.14483}}, 2024.

\bibitem[Gopakumar et~al.(2025)Gopakumar, Gray, Zanisi, Nunn, Giles, Kusner,
  Pamela, and Deisenroth]{gopakumar2025calibrated}
Vignesh Gopakumar, Ander Gray, Lorenzo Zanisi, Timothy Nunn, Daniel Giles, Matt
  Kusner, Stanislas Pamela, and Marc~Peter Deisenroth.
\newblock Calibrated Physics-Informed Uncertainty Quantification.
\newblock In {\em\protect\JournalTitle{International Conference on Machine
  Learning}}, 2025.

\bibitem[Gopakumar et~al.(2026)Gopakumar, Gray, Oskarsson, Zanisi, Giles,
  Kusner, Pamela, and Peter~Deisenroth]{gopakumar2026uncertainty}
Vignesh Gopakumar, Ander Gray, Joel Oskarsson, Lorenzo Zanisi, Daniel Giles,
  Matt~J Kusner, Stanislas Pamela, and Marc Peter~Deisenroth.
\newblock Uncertainty quantification of surrogate models using conformal
  prediction.
\newblock {\em\protect\JournalTitle{Machine Learning: Science and Technology}},
  7\penalty0 (1):\penalty0 015025, 2026.

\bibitem[Gray et~al.(2025)Gray, Gopakumar, Rousseau, and
  Destercke]{gray2025guaranteed}
Ander Gray, Vignesh Gopakumar, Sylvain Rousseau, and Sebastien Destercke.
\newblock Guaranteed Prediction Sets for Functional Surrogate Models.
\newblock In {\em\protect\JournalTitle{Conference on Uncertainty in Artificial
  Intelligence}}, pages 1569--1585, 2025.

\bibitem[Gupta et~al.(2021)Gupta, Xiao, and Bogdan]{gupta2021multiwavelet}
Gaurav Gupta, Xiongye Xiao, and Paul Bogdan.
\newblock Multiwavelet-based operator learning for differential equations.
\newblock In {\em\protect\JournalTitle{Advances in Neural Information
  Processing Systems}}, volume~34, 2021.

\bibitem[Hao et~al.(2023)Hao, Wang, Su, Ying, Dong, Liu, Cheng, Song, and
  Zhu]{hao2023gnot}
Zhongkai Hao, Zhengyi Wang, Hang Su, Chengyang Ying, Yinpeng Dong, Songming
  Liu, Ze~Cheng, Jun Song, and Jun Zhu.
\newblock {GNOT}: A general neural operator transformer for operator learning.
\newblock In {\em\protect\JournalTitle{International Conference on Machine
  Learning}}, 2023.

\bibitem[He et~al.(2024)He, Liu, and Xu]{he2024mgno}
Juncai He, Xinliang Liu, and Jinchao Xu.
\newblock MgNO: Efficient parameterization of linear operators via multigrid.
\newblock In {\em\protect\JournalTitle{International Conference on Learning
  Representations}}, volume 2024, pages 53409--53428, 2024.

\bibitem[He et~al.(2026)He, Jiang, Xiao, Xu, and Li]{he2026survey}
Wenchong He, Zhe Jiang, Tingsong Xiao, Zelin Xu, and Yukun Li.
\newblock A survey on uncertainty quantification methods for deep learning.
\newblock {\em\protect\JournalTitle{ACM Computing Surveys}}, 58\penalty0
  (7):\penalty0 1--35, 2026.

\bibitem[Herde et~al.(2024)Herde, Raoni{\'c}, Rohner, K{\"a}ppeli, Molinaro,
  De~Bezenac, and Mishra]{herde2024poseidon}
Maximilian Herde, Bogdan Raoni{\'c}, Tobias Rohner, Roger K{\"a}ppeli, Roberto
  Molinaro, Emmanuel De~Bezenac, and Siddhartha Mishra.
\newblock Poseidon: Efficient foundation models for {PDEs}.
\newblock In {\em\protect\JournalTitle{Advances in Neural Information
  Processing Systems}}, volume~37, pages 72525--72624, 2024.

\bibitem[Karniadakis et~al.(2021)Karniadakis, Kevrekidis, Lu, Perdikaris, Wang,
  and Yang]{karniadakis2021physics}
George~Em Karniadakis, Ioannis~G Kevrekidis, Lu~Lu, Paris Perdikaris, Sifan
  Wang, and Liu Yang.
\newblock Physics-informed machine learning.
\newblock {\em\protect\JournalTitle{Nature Reviews Physics}}, 3\penalty0
  (6):\penalty0 422--440, 2021.

\bibitem[Kossaifi et~al.(2025)Kossaifi, Kovachki, Li, Pitt, Liu-Schiaffini,
  George, Bonev, Azizzadenesheli, Berner, Duruisseaux, and
  Anandkumar]{kossaifi2025librarylearningneuraloperators}
Jean Kossaifi, Nikola Kovachki, Zongyi Li, David Pitt, Miguel Liu-Schiaffini,
  Robert~J. George, Boris Bonev, Kamyar Azizzadenesheli, Julius Berner,
  Valentin Duruisseaux, and Anima Anandkumar.
\newblock A library for learning neural operators.
\newblock {\em\protect\JournalTitle{arXiv preprint arXiv:2412.10354}}, 2025.

\bibitem[Kovachki et~al.(2021)Kovachki, Lanthaler, and
  Mishra]{kovachki2021universal}
Nikola Kovachki, Samuel Lanthaler, and Siddhartha Mishra.
\newblock On universal approximation and error bounds for Fourier neural
  operators.
\newblock {\em\protect\JournalTitle{Journal of Machine Learning Research}},
  22\penalty0 (290):\penalty0 1--76, 2021.

\bibitem[Kovachki et~al.(2023)Kovachki, Li, Liu, Azizzadenesheli, Bhattacharya,
  Stuart, and Anandkumar]{kovachki2023neural}
Nikola Kovachki, Zongyi Li, Burigede Liu, Kamyar Azizzadenesheli, Kaushik
  Bhattacharya, Andrew Stuart, and Anima Anandkumar.
\newblock Neural operator: Learning maps between function spaces with
  applications to pdes.
\newblock {\em\protect\JournalTitle{Journal of Machine Learning Research}},
  24\penalty0 (89):\penalty0 1--97, 2023.

\bibitem[Kovachki et~al.(2024)Kovachki, Lanthaler, and
  Stuart]{kovachki2024operator}
Nikola~B Kovachki, Samuel Lanthaler, and Andrew~M Stuart.
\newblock Operator learning: Algorithms and analysis.
\newblock {\em\protect\JournalTitle{Handbook of Numerical Analysis}},
  25:\penalty0 419--467, 2024.

\bibitem[Kurth et~al.(2023)Kurth, Subramanian, Harrington, Pathak, Mardani,
  Hall, Miele, Kashinath, and Anandkumar]{pathak2022fourcastnet}
Thorsten Kurth, Shashank Subramanian, Peter Harrington, Jaideep Pathak, Morteza
  Mardani, David Hall, Andrea Miele, Karthik Kashinath, and Anima Anandkumar.
\newblock {FourCastNet: Accelerating Global High-Resolution Weather Forecasting
  Using Adaptive Fourier Neural Operators}.
\newblock In {\em\protect\JournalTitle{Proceedings of the Platform for Advanced
  Scientific Computing Conference}}. ACM, 2023.

\bibitem[Lakshminarayanan et~al.(2017)Lakshminarayanan, Pritzel, and
  Blundell]{lakshminarayanan2017simple}
Balaji Lakshminarayanan, Alexander Pritzel, and Charles Blundell.
\newblock Simple and scalable predictive uncertainty estimation using deep
  ensembles.
\newblock In {\em\protect\JournalTitle{Advances in Neural Information
  Processing Systems}}, volume~30, 2017.

\bibitem[Lam et~al.(2023)Lam, Sanchez-Gonzalez, Willson, Wirnsberger,
  Fortunato, Alet, Ravuri, Ewalds, Eaton-Rosen, Hu, et~al.]{lam2023graphcast}
Remi Lam, Alvaro Sanchez-Gonzalez, Matthew Willson, Peter Wirnsberger, Meire
  Fortunato, Ferran Alet, Suman Ravuri, Timo Ewalds, Zach Eaton-Rosen, Weihua
  Hu, et~al.
\newblock Learning skillful medium-range global weather forecasting.
\newblock {\em\protect\JournalTitle{Science}}, 382\penalty0 (6677):\penalty0
  1416--1421, 2023.

\bibitem[Lanthaler et~al.(2022)Lanthaler, Mishra, and
  Karniadakis]{lanthaler2022error}
Samuel Lanthaler, Siddhartha Mishra, and George~E Karniadakis.
\newblock Error estimates for deeponets: {A} deep learning framework in
  infinite dimensions.
\newblock {\em\protect\JournalTitle{Transactions of Mathematics and its
  Applications}}, 6\penalty0 (1):\penalty0 tnac001, 2022.

\bibitem[Li et~al.(2020)Li, Kovachki, Azizzadenesheli, Liu, Stuart,
  Bhattacharya, and Anandkumar]{li2020multipole}
Zongyi Li, Nikola Kovachki, Kamyar Azizzadenesheli, Burigede Liu, Andrew
  Stuart, Kaushik Bhattacharya, and Anima Anandkumar.
\newblock Multipole graph neural operator for parametric partial differential
  equations.
\newblock In {\em\protect\JournalTitle{Advances in Neural Information
  Processing Systems}}, volume~33, pages 6755--6766, 2020.

\bibitem[Li et~al.(2021)Li, Kovachki, Azizzadenesheli, liu, Bhattacharya,
  Stuart, and Anandkumar]{li2021fourier}
Zongyi Li, Nikola~Borislavov Kovachki, Kamyar Azizzadenesheli, Burigede liu,
  Kaushik Bhattacharya, Andrew Stuart, and Anima Anandkumar.
\newblock Fourier Neural Operator for Parametric Partial Differential
  Equations.
\newblock In {\em\protect\JournalTitle{International Conference on Learning
  Representations}}, 2021.

\bibitem[Lu et~al.(2021)Lu, Jin, Pang, Zhang, and Karniadakis]{lu2021deeponet}
Lu~Lu, Pengzhan Jin, Guofei Pang, Zhongqiang Zhang, and George~Em Karniadakis.
\newblock Learning nonlinear operators via {DeepONet} based on the universal
  approximation theorem of operators.
\newblock {\em\protect\JournalTitle{Nature Machine Intelligence}}, 3:\penalty0
  218--229, 2021.

\bibitem[Ma et~al.(2024)Ma, Pitt, Azizzadenesheli, and
  Anandkumar]{ma2024calibrated}
Ziqi Ma, David Pitt, Kamyar Azizzadenesheli, and Anima Anandkumar.
\newblock Calibrated Uncertainty Quantification for Operator Learning via
  Conformal Prediction.
\newblock {\em\protect\JournalTitle{Transactions on Machine Learning
  Research}}, 2024.

\bibitem[Magnani et~al.(2025{\natexlab{a}})Magnani, Kr{\"a}mer, Eschenhagen,
  Rosasco, and Hennig]{magnani2022abno}
Emilia Magnani, Nicholas Kr{\"a}mer, Runa Eschenhagen, Lorenzo Rosasco, and
  Philipp Hennig.
\newblock Approximate Bayesian Neural Operators: Uncertainty Quantification for
  Parametric {PDE}s.
\newblock {\em\protect\JournalTitle{Transactions on Machine Learning
  Research}}, 2025{\natexlab{a}}.

\bibitem[Magnani et~al.(2025{\natexlab{b}})Magnani, Pf{\"o}rtner, Weber, and
  Hennig]{magnani2024luno}
Emilia Magnani, Marvin Pf{\"o}rtner, Tobias Weber, and Philipp Hennig.
\newblock Linearization turns neural operators into function-valued Gaussian
  processes.
\newblock In {\em\protect\JournalTitle{International Conference on Machine
  Learning}}, 2025{\natexlab{b}}.

\bibitem[Millard et~al.(2025)Millard, Lindemann, and Baheri]{millard2025split}
David Millard, Lars Lindemann, and Ali Baheri.
\newblock Split conformal prediction in the function space with neural
  operators.
\newblock {\em\protect\JournalTitle{arXiv preprint arXiv:2509.04623}}, 2025.

\bibitem[Mouli et~al.(2024)Mouli, Maddix, Alizadeh, Gupta, Stuart, Mahoney, and
  Wang]{mouli2024using}
S~Chandra Mouli, Danielle~C. Maddix, Shima Alizadeh, Gaurav Gupta, Andrew
  Stuart, Michael~W. Mahoney, and Bernie Wang.
\newblock Using Uncertainty Quantification to Characterize and Improve
  Out-of-Domain Learning for {PDE}s.
\newblock In {\em\protect\JournalTitle{International Conference on Machine
  Learning}}, 2024.

\bibitem[Moya et~al.(2025)Moya, Mollaali, Zhang, Lu, and
  Lin]{moya2025conformalized}
Christian Moya, Amirhossein Mollaali, Zecheng Zhang, Lu~Lu, and Guang Lin.
\newblock {Conformalized-deeponet: A distribution-free framework for
  uncertainty quantification in deep operator networks}.
\newblock {\em\protect\JournalTitle{Physica D: Nonlinear Phenomena}},
  471:\penalty0 134418, 2025.

\bibitem[Psaros et~al.(2023)Psaros, Meng, Zou, Guo, and
  Karniadakis]{psaros2023uncertainty}
Apostolos~F Psaros, Xuhui Meng, Zongren Zou, Ling Guo, and George~Em
  Karniadakis.
\newblock Uncertainty quantification in scientific machine learning: Methods,
  metrics, and comparisons.
\newblock {\em\protect\JournalTitle{Journal of Computational Physics}},
  477:\penalty0 111902, 2023.

\bibitem[Rahman et~al.(2023)Rahman, Ross, and Azizzadenesheli]{rahman2023uno}
Md~Ashiqur Rahman, Zachary~E. Ross, and Kamyar Azizzadenesheli.
\newblock {U-NO}: U-shaped neural operators.
\newblock {\em\protect\JournalTitle{Transactions on Machine Learning
  Research}}, 2023.

\bibitem[Raonic et~al.(2023)Raonic, Molinaro, De~Ryck, Rohner, Bartolucci,
  Alaifari, Mishra, and De~B{\'e}zenac]{raonic2023convolutional}
Bogdan Raonic, Roberto Molinaro, Tim De~Ryck, Tobias Rohner, Francesca
  Bartolucci, Rima Alaifari, Siddhartha Mishra, and Emmanuel De~B{\'e}zenac.
\newblock Convolutional neural operators for robust and accurate learning of
  {PDEs}.
\newblock In {\em\protect\JournalTitle{Advances in Neural Information
  Processing Systems}}, volume~36, pages 77187--77200, 2023.

\bibitem[Shafer and Vovk(2008)]{shafer2008tutorial}
Glenn Shafer and Vladimir Vovk.
\newblock A tutorial on conformal prediction.
\newblock {\em\protect\JournalTitle{Journal of Machine Learning Research}},
  9:\penalty0 371--421, 2008.

\bibitem[Vovk(2012)]{vovk2012conditionalvalidityinductiveconformal}
Vladimir Vovk.
\newblock Conditional validity of inductive conformal predictors.
\newblock In {\em\protect\JournalTitle{Asian Conference on Machine Learning}},
  pages 475--490. PMLR, 2012.

\bibitem[Vovk et~al.(2005)Vovk, Gammerman, and Shafer]{vovk2005algorithmic}
Vladimir Vovk, Alex Gammerman, and Glenn Shafer.
\newblock \emph{Algorithmic Learning in a Random World}.
\newblock Springer, 2005.

\bibitem[Wang et~al.(2026{\natexlab{a}})Wang, Xin, Wang, Yang, Zha, Jiang,
  et~al.]{wang2026mixture}
Hong Wang, Haiyang Xin, Jie Wang, Xuanze Yang, Fei Zha, Yan Jiang, et~al.
\newblock Mixture-of-experts operator transformer for large-scale pde
  pre-training.
\newblock In {\em\protect\JournalTitle{Advances in Neural Information
  Processing Systems}}, volume~38, pages 31498--31527, 2026{\natexlab{a}}.

\bibitem[Wang et~al.(2026{\natexlab{b}})Wang, Gang, and Deng]{wang2026operator}
Weinan Wang, Bowen Gang, and Hao Deng.
\newblock {Operator learning for the 2D incompressible Navier-Stokes equations:
  a conformal prediction approach in the data-scarce regime}.
\newblock {\em\protect\JournalTitle{arXiv preprint arXiv:2606.08654}},
  2026{\natexlab{b}}.

\bibitem[Yu et~al.(2026)Yu, Ho, and Wang]{yu2026conformal}
Yifan Yu, Cheuk~Hin Ho, and Yangshuai Wang.
\newblock A conformal prediction framework for uncertainty quantification in
  physics-informed neural networks.
\newblock {\em\protect\JournalTitle{Journal of Computational Physics}},
  561:\penalty0 114979, 2026.

\bibitem[Zou et~al.(2024)Zou, Meng, Psaros, and Karniadakis]{zou2024neuraluq}
Zongren Zou, Xuhui Meng, Apostolos~F Psaros, and George~E Karniadakis.
\newblock {NeuralUQ: A comprehensive library for uncertainty quantification in
  neural differential equations and operators}.
\newblock {\em\protect\JournalTitle{SIAM Review}}, 66\penalty0 (1):\penalty0
  161--190, 2024.

\end{thebibliography}

\appendix

\section{Technical Proofs and Auxiliary Results}

\subsection{Independence of the representatives}

We show that the measure of the containment set defined in \cref{sec:setup} does not depend on the choice of representatives for the operators $\mathcal{G}^\dagger$, $\hat{\mathcal{G}}$, and $\hat{\mathcal{E}}$.

\begin{lemma}[Representative independence]\label{lemma:rep-indep}
  Let $(g^\dagger_1,\hat g_1,\hat e_1)$ and $(g^\dagger_2,\hat g_2,\hat e_2)$ be two families of product-measurable representatives of $(\mathcal{G}^\dagger, \hat{\mathcal{G}}, \hat{\mathcal{E}})$, with corresponding prediction bands $C^1_\lambda$ and $C^2_\lambda$. Then for every $a \in \mathcal{A}$ and $\lambda \in [0,\infty]$,
  \[
    \nu\bigl(\{x \in D : g^\dagger_1(a,x) \in C^1_\lambda(a)(x)\}\bigr)
    = \nu\bigl(\{x \in D : g^\dagger_2(a,x) \in C^2_\lambda(a)(x)\}\bigr).
  \]
\end{lemma}
\begin{proof}
  Fix $a$ and $\lambda$; the case $\lambda = \infty$ is trivial since $C^i_\infty(a)(x) = \mathbb{R}^{d_u}$ for $i=1,2$.
  For $\lambda < \infty$ and $i = 1,2$, set $E_i := \{x \in D : g^\dagger_i(a,x) \in C^i_\lambda(a)(x)\} = \{x : \|g^\dagger_i(a,x) - \hat g_i(a,x)\|_2 \le \lambda\,\hat e_i(a,x)\}$.
  Since each factor is $\mathcal{B}(\mathcal{A}) \otimes \mathcal{B}(D)$-measurable, $E_i$ is an $a$-section of a product-measurable set, hence $E_i \in \mathcal{B}(D)$ and $\nu(E_i)$ is defined.
  Because both families represent the same equivalence classes, for each $i$ the equality $g^\dagger_1(a,\cdot) = g^\dagger_2(a,\cdot)$ holds $\nu$-a.e., and likewise for $\hat g$ and $\hat e$. These three a.e.\ equalities exclude a common $\nu$-null set $N$; on $D \setminus N$ the defining inequalities of $E_1$ and $E_2$ involve identical quantities, so $E_1 \triangle E_2 \subseteq N$ and $\nu(E_1) = \nu(E_2)$.
\end{proof}

\subsection{Conformal calibration for a generic residual field}\label{sec:general proof}

This section provides a generic calibration argument that applies to any residual field over a probability space $(X,\mathcal{X},\pi)$, including the continuum and discrete residual fields in \cref{sec:theory}. In both cases, there is a nonnegative, jointly measurable field measuring the discrepancy between prediction and truth, a probability measure on the domain over which containment is assessed, a scalar nonconformity score set as a quantile of that field, and a split-conformal calibration procedure that turns exchangeability of the scores into a coverage guarantee. The argument proceeds in four steps. First, we convert a residual field into a scalar score by taking the smallest threshold that contains at least a $1-\gamma$ fraction of the evaluation space. Second, we establish the measurability of this score and show that thresholding it is equivalent to achieving the desired containment fraction. Third, an exchangeable-rank argument calibrates the $k$th order statistic of the scores and yields marginal coverage at least $1-\alpha$. Finally, we show that the same guarantee remains valid when calibration uses a computable score that upper-bounds an ideal score.

Let $\gamma,\alpha\in(0,1)$, and $A_1,\dots,A_{n+1}$ be exchangeable with common law $\mu$ on the separable Banach space $\mathcal{A}$, for $n\geq 1$. Let $r:\mathcal{A}\times X \rightarrow [0, \infty]$ be a nonnegative $\mathcal{B}(\mathcal{A}) \otimes \mathcal{X}$-measurable map, which we call a \emph{residual field} over $(X, \mathcal{X}, \pi)$. The field is evaluated at $x\in X$ for each $a\in\mathcal{A}$, and measures the discrepancy between a prediction and a ground truth. For $a\in\mathcal{A}$ and $\lambda\in[0,\infty]$, we say that the field is contained at $x$ if $r(a,x)\leq\lambda$, and denote by
\begin{equation}
  w^r_\lambda(a)=\pi\bigl(\{x\in X:r(a,x)\leq\lambda\}\bigr)
\end{equation}
the corresponding \emph{containment fraction}. The quantity $w^r_\lambda(a)$ measures the fraction (with respect to $\pi$) of the evaluation space $X$ for which the residual field is contained by $\lambda$. The next lemma ensures that the containment fraction $w^r_\lambda(a)$ is monotone and right-continuous in $\lambda$ and measurable in $a$.

\begin{lemma}[Regularity and measurability]\label{lemma:w borel}
  For each $a\in\mathcal{A}$, the map $\lambda\mapsto w^r_\lambda(a)$ is nondecreasing and right-continuous, with $w^r_\infty(a)=1$. Moreover, for each $\lambda\in[0,\infty]$, the map $a\mapsto w^r_\lambda(a)$ is Borel measurable.
\end{lemma}
\begin{proof}
  Let $0 \le \lambda_1 \le \lambda_2$. Then
  $\{x : r(a,x) \le \lambda_1\} \subseteq \{x : r(a,x) \le \lambda_2\}$, and
  monotonicity of $\pi$ implies that $w^r_{\lambda_1}(a) \le w^r_{\lambda_2}(a)$. For the right-continuity, let $\lambda_0 \ge 0$ and $\lambda_m \downarrow
    \lambda_0$. The sets $\{x : r(a,x) \le \lambda_m\}$ are nonincreasing in $m$,
  and $\bigcap_{m \in \mathbb{N}} \{x : r(a,x) \le \lambda_m\}
    = \{x : r(a,x) \le \lambda_0\}$: the inclusion $\supseteq$ holds since
  $\lambda_m \ge \lambda_0$ for all $m$. Conversely, if $x$ lies in every set
  then $r(a,x) \le \lambda_m$ for all $m$, and letting $m \to \infty$ gives
  $r(a,x) \le \inf_m \lambda_m = \lambda_0$. Continuity from above of the
  probability measure $\pi$ yields
  $\lim_{m \to \infty} w^r_{\lambda_m}(a) = w^r_{\lambda_0}(a)$. Finally, every $x \in X$ satisfies $r(a,x) \le \infty$, so
  $w^r_\infty(a) = \pi(X) = 1$.

  For the second part, we note that $r$ is $\mathcal{B}(\mathcal{A}) \otimes \mathcal{X}$-measurable as a map into $[0,\infty]$, so $(a,x) \mapsto \mathbf{1}\{r(a,x) \le \lambda\}$ is a nonnegative product-measurable function. As a probability measure, $\pi$ is $\sigma$-finite, so Tonelli's theorem applies, and the partial integral
  $a \mapsto \int_X \mathbf{1}\{r(a,x) \le \lambda\}\, d\pi(x) = w^r_\lambda(a)$
  is $\mathcal{B}(\mathcal{A})$-measurable.
\end{proof}

Given $a\in\mathcal{A}$, we define the \emph{nonconformity score} $s^r(a)$ as the infimum of the scaling factors $\lambda$ for which the containment fraction $w^r_\lambda(a)$ is greater than $1-\gamma$:
\begin{equation}
  s^r(a) = \inf \{ \lambda \ge 0: w^r_\lambda(a) \ge 1-\gamma \},
\end{equation}
where $\inf \emptyset := +\infty$ by convention.
The right-continuity of $\lambda \mapsto w^r_\lambda(a)$ is crucial and ensures that the infimum in the definition of $s^r(a)$ is attained, so that $s^r(a) \le \lambda$ if and only if $w^r_\lambda(a) \ge 1-\gamma$, as shown by the following lemma.

\begin{lemma}[Score-containment equivalence]\label{lemma:score-containment equivalence}
  For every $a \in \mathcal{A}$ and $\lambda \in [0, \infty]$, $s^r(a) \le \lambda$ if and only if $w^r_\lambda(a) \ge 1 - \gamma$.
\end{lemma}
\begin{proof}
  If $w^r_\lambda(a)\geq1-\gamma$, then $s^r(a)\leq\lambda$ by definition (and trivially when $\lambda=\infty$). Conversely, suppose $s^r(a)\leq\lambda$. If $s^r(a)=\infty$, then $\lambda=\infty$ and $w^r_\lambda(a)=1$. Otherwise choose $t_m\downarrow s^r(a)$ with $w^r_{t_m}(a)\geq1-\gamma$. Following \cref{lemma:w borel}, the right-continuity of $\lambda\mapsto w_\lambda^r(a)$ implies that $w^r_{s^r(a)}(a) = \lim_m w^r_{t_m}(a) \ge 1-\gamma$; in particular the infimum is attained. Monotonicity then gives
  $w^r_\lambda(a) \ge w^r_{s^r(a)}(a) \ge 1-\gamma$.
\end{proof}

Set $k=\lceil(1-\alpha)(n+1)\rceil$ and for an $n$-tuple $v\in[0,\infty]^n$, denote by $T_k(v)$ its $k$th smallest entry. Split conformal calibration applies this order statistic to the calibration scores.

\begin{definition}[Calibrated scaling factor]\label{def:calib value}
  The calibration map $\hat\lambda^r : \mathcal{A}^n \to [0,\infty]$ is defined for $(a_1,\ldots,a_n)\in \mathcal{A}^n$ as $
    \hat\lambda^r(a_1,\dots,a_n) = T_k\bigl(s^r(a_1),\dots,s^r(a_n)\bigr)$ if $k \le n$, and $\hat\lambda^r(a_1,\dots,a_n) = \infty$ otherwise.
\end{definition}

We abbreviate $\hat\lambda^r := \hat\lambda^r(A_1,\dots,A_n)$ when no confusion arises, and first connect the geometric containment construction to the scalar conformal calibration. \cref{def:calib value} determines the scaling factor provided by the calibration procedure. In the rest of the section, we show that it provides the desired coverage guarantee, i.e., that the residual field is contained on at least a $1-\gamma$ fraction of $X$ with probability at least $1-\alpha$. The convention $\hat\lambda^r=\infty$ guarantees formal coverage when the requested quantile exceeds the calibration sample, but gives a vacuous band. The following mild condition ensures finite scores, and hence a finite calibrated scaling when $k\leq n$.
\begin{lemma}[Finiteness of the score]\label{lemma:finite score}
  Let $a\in\mathcal{A}$ such that $\pi(\{x : r(a,x) < \infty\}) = 1$, then $s^r(a) < \infty$. Moreover, if this holds for all $a\in\mathcal{A}$ outside a $\mu$-null set and $k \le n$, then $\hat\lambda^r < \infty$ $\mu^{\otimes n}$-almost surely.
\end{lemma}
\begin{proof}
  Fix $a\in\mathcal{A}$ with $\pi(\{x : r(a,x) < \infty\}) = 1$. For $m\in\mathbb{N}$, the sets
  $\{x : r(a,x) \le m\}$, $m \in \mathbb{N}$ increase to
  $\{x : r(a,x) < \infty\}$, so by continuity from below of $\pi$, we have
  \[
    w^r_m(a) \xrightarrow[m\to\infty]{} \pi(\{x : r(a,x) < \infty\}) = 1 > 1-\gamma,
  \]
  where the inequality is strict because $\gamma > 0$. Hence, there exists $m\in\mathbb{N}$ such that $w^r_m(a) \ge 1-\gamma$, which implies that $s^r(a) \le m < \infty$ by \cref{lemma:score-containment equivalence}.

  For the second statement, suppose the assumption holds for all $a\in\mathcal{A}$ outside a $\mu$-null set and $k \le n$. Then $s^r(A_i) < \infty$ a.s.\ for each $i \le n$ by the first part of the lemma. Therefore, $\max_{i \le n} s^r(A_i) < \infty$ a.s., and since $T_k(v) \le \max_i v_i$ for $k \le n$, we have $
    \hat\lambda^r = T_k\bigl(s^r(A_1),\dots,s^r(A_n)\bigr)
    \le \max_{i \le n} s^r(A_i) < \infty$ $\mu^{\otimes n}$-almost surely.
\end{proof}

To apply the probabilistic arguments below, we must verify that the score, the calibrated scaling factor, and their comparison event are measurable. We formulate this for a general Borel score $\sigma$ so that the result also applies to the dominating-score calibration in \cref{cor:dominating score}.
\begin{lemma}[Event measurability]\label{lemma:event-meas}
  The score $s^r$ and the calibration map $\hat\lambda^r$ are Borel measurable. More generally, for any Borel score $\sigma:\mathcal{A}\to[0,\infty]$, let $\hat\lambda^\sigma$ be the calibration map obtained by replacing $s^r$ with $\sigma$ in \cref{def:calib value}. Then $\hat\lambda^\sigma$ is measurable, and
  \[
    E_\sigma = \bigl\{(a_1,\dots,a_{n+1})\in\mathcal{A}^{n+1}:
    s^r(a_{n+1})\leq\hat\lambda^\sigma(a_1,\dots,a_n)\bigr\}
    \in\mathcal{B}(\mathcal{A})^{\otimes(n+1)}.
  \]
\end{lemma}
\begin{proof}
  By combining \cref{lemma:score-containment equivalence,lemma:w borel}, we find that the sublevel sets of $s^r$ are Borel:
  $\{a\in\mathcal{A}:s^r(a)\leq\lambda\}=\{a\in\mathcal{A}:w^r_\lambda(a)\geq1-\gamma\}$ is Borel for every finite $\lambda$ and equal to $\mathcal{A}$ if $\lambda=\infty$, hence $s^r$ is Borel. Moreover, if $k\leq n$, then
  \[
    \{v\in[0,\infty]^n:T_k(v)\leq t\}
    =\bigcup_{\substack{I\subseteq\{1,\dots,n\}\\|I|=k}}
    \bigcap_{i\in I}\{v\in[0,\infty]^n:v_i\leq t\}.
  \]
  Each set $\{v\in[0,\infty]^n:v_i\leq t\}$ is Borel because the coordinate projection $v\mapsto v_i$ is continuous. Thus the right-hand side is Borel for every $t$ as a finite union of finite intersections of measurable sets, so $T_k$ is measurable. Since $\sigma$ is Borel, the coordinatewise score map $(a_1, \dots, a_n) \mapsto (\sigma(a_1), \dots, \sigma(a_n))$ is
  $\mathcal{B}(\mathcal{A})^{\otimes n}$-measurable, since a map into a product space is measurable if and only if each coordinate map is, and each coordinate is $\sigma$ composed with a coordinate projection. Therefore, the composition $\hat\lambda^\sigma = T_k \circ (\sigma \times \dots \times \sigma)$ is measurable. Finally, this implies that the map $(a_1,\dots,a_{n+1}) \mapsto \bigl(s^r(a_{n+1}), \hat\lambda^\sigma(a_1,\dots,a_n)\bigr) \in [0,\infty]^2$ is $\mathcal{B}(\mathcal{A})^{\otimes(n+1)}$-measurable, and $\{(u,v) \in [0,\infty]^2 : u \le v\}$ is closed, hence Borel, in $[0,\infty]^2$. Then its preimage $E_\sigma$ lies in $\mathcal{B}(\mathcal{A})^{\otimes(n+1)}$.
\end{proof}

The following lemma formalizes the decomposition of the marginal coverage probability in \eqref{eq:marginal guarantee} into an average of conditional coverage probabilities. Conditioning on the calibration set freezes $\hat{\lambda}$ at a fixed number, so the only randomness left is the draw of the test input. The conditional probability of coverage is then simply the proportion of inputs whose score falls below that number, that is, the score distribution function evaluated at $\hat{\lambda}$. Thus $F(\hat{\lambda}^r)$ varies across random draws of the calibration set, but is fixed once a particular calibration set is realized. The marginal guarantee controls the mean of this quantity over calibration sets; it does not require the conditional coverage for every realized calibration set to be at least $1-\alpha$.

\begin{proposition}[Conditional coverage identity]\label{lemma:tower}
  Suppose $(A_1,\dots,A_n)$ is independent of $A_{n+1} \sim \mu$, and for $\lambda \in [0,\infty]$, denote by $F(\lambda) = \mu(\{a : s^r(a) \le \lambda\})$, the distribution function of the score $s^r(A_{n+1})$. Then, for almost every realization of the calibration sample $(A_1,\dots,A_n)$ with respect to its law (which is $\mu^{\otimes n}$ when the calibration inputs are i.i.d.),
  \[
    \mathbb{P}\bigl( w^r_{\hat\lambda^r}(A_{n+1}) \ge 1-\gamma \,\bigm|\,
    A_1,\dots,A_n \bigr) = F(\hat\lambda^r),\]
  hence $
    \mathbb{P}\bigl(s^r(A_{n+1})\leq\hat\lambda^r\bigr)
    = \mathbb{E}\bigl[F(\hat\lambda^r)\bigr]$.
\end{proposition}

\begin{proof}
  Let $V = (A_1,\dots,A_n)$ and define $h:\mathcal{A}^n\times\mathcal{A}\to\{0,1\}$ by
  $h(v,w) := \mathbf{1}\{s^r(w) \le \hat\lambda^r(v)\}$ for $v\in\mathcal{A}^n$ and $w\in\mathcal{A}$. This function is bounded and jointly measurable by \cref{lemma:event-meas}. For every $(v,w)\in\mathcal{A}^n\times\mathcal{A}$, \cref{lemma:score-containment equivalence} applied with $\lambda=\hat\lambda^r(v)$ implies that $w^r_{\hat\lambda^r(v)}(w)\geq1-\gamma$ if and only if $s^r(w)\leq\hat\lambda^r(v)$. Therefore, we have the following equivalence of events:
  \[
    \{w^r_{\hat\lambda^r(V)}(A_{n+1})\geq1-\gamma\}
    =\{s^r(A_{n+1})\leq\hat\lambda^r(V)\}
    =\{h(V,A_{n+1})=1\}.
  \]
  Since $V$ and $A_{n+1}$ are independent and $h$ is bounded and measurable, the standard independence lemma for conditional expectations (Fubini plus a monotone class argument; \citealt[Ex.~4.1.7]{durrett2019probability}) implies that the following equality holds almost surely,
  \[
    \mathbb{E}\bigl[h(V,A_{n+1}) \mid V\bigr]
    = \int_{\mathcal{A}} h(V,w)\, \mu(\mathrm{d}w)
    = \mu\bigl(\{w : s^r(w) \le \hat\lambda^r(V)\}\bigr)
    = F(\hat\lambda^r).
  \]
  The second identity is the tower property of conditional expectation
  $\mathbb{P}(E_{s^r}) = \mathbb{E}\bigl[\mathbb{E}[\mathbf{1}_{E_{s^r}} \mid
        A_1,\dots,A_n]\bigr]$, where $E_{s^r} = \{s^r(A_{n+1})\leq\hat\lambda^r\}$ is the event defined in \cref{lemma:event-meas}.
\end{proof}

We can now apply an exchangeable-rank argument to the scalar scores and translate the result back to the containment score using \cref{lemma:score-containment equivalence}.

\begin{theorem}[Marginal coverage guarantee]\label{thm:general guarantee}
  Let $r:\mathcal{A}\times X \rightarrow [0, \infty]$ be a residual field over $(X,\mathcal{X}, \pi)$, select $\gamma,\alpha \in (0,1)$ as the spatial and probability tolerances, and suppose $A_1,\dots,A_{n+1}$ are exchangeable. Let $k = \lceil (n+1)(1-\alpha) \rceil$ and consider the scaling factor $\hat{\lambda}^r$ defined as the $k$th smallest value among the scores $s^r(A_1),\dots,s^r(A_n)$ if $k \le n$, and $\hat{\lambda}^r = \infty$ otherwise. Then,
  \[
    \mathbb{P}\bigl( w^r_{\hat\lambda^r}(A_{n+1}) \ge 1-\gamma \bigr)
    \ge \frac{k}{n+1} \ge 1-\alpha ,
  \]
  that is the coverage is greater than $1-\alpha$. Moreover, if the $A_i$ are i.i.d.\ and the law of $s^r(A_1)$ is atomless, then the random variable $\mathbb{P}\bigl( w^r_{\hat\lambda^r}(A_{n+1}) \ge 1-\gamma \,\bigm|\,
    A_1,\dots,A_n \bigr)$ follows a $\mathrm{Beta}(k,\, n+1-k)$ distribution, in particular
  \[
    \mathbb{P}\bigl( w^r_{\hat\lambda^r}(A_{n+1}) \ge 1-\gamma \bigr)
    = \mathbb{E}[\mathbb{P}\bigl( w^r_{\hat\lambda^r}(A_{n+1}) \ge 1-\gamma \,\bigm|\,
      A_1,\dots,A_n \bigr)]=  \frac{k}{n+1} \le 1-\alpha+\frac{1}{n+1} .
  \]
\end{theorem}
\begin{proof}
  Following \cref{lemma:event-meas} with $\sigma = s^r$, the set $E_{s^r}$ defined in the lemma's statement is an event in $\mathcal{B}(\mathcal{A})^{\otimes(n+1)}$ which coincides with $\{w^r_{\hat\lambda^r}(A_{n+1}) \ge 1-\gamma\}$ by \cref{lemma:score-containment equivalence}. We define $S_i := s^r(A_i)$ so that $E_{s^r} = \{S_{n+1} \le \hat\lambda^r\}$. If $k > n$ then $\hat\lambda^r = \infty$ and $\mathbb{P}(E_{s^r}) = 1$
  trivially. Otherwise, $\hat\lambda^r = T_k(S_1,\dots,S_n)$, the $k$th smallest value
  of the calibration scores. The variables $S_1,\dots,S_{n+1}$ are exchangeable,
  being the image of an exchangeable vector under the fixed Borel map $s^r$
  applied coordinatewise. The split conformal coverage theorem
  \citep[Fact~2.15(i) \& Prop.~3.12]{angelopoulos2024conformalbook} implies that
  \[
    \mathbb{P}\bigl(\{w^r_{\hat\lambda^r}(A_{n+1}) \ge 1-\gamma\}\bigr)=\mathbb{P}\bigl(S_{n+1} \le T_k(S_1,\dots,S_n)\bigr) \ge \frac{k}{n+1}
    \ge 1-\alpha .
  \]
  For the second claim, if the $A_i$ are i.i.d.\ then $S_1,\dots,S_{n+1}$ are
  i.i.d., and since their law is atomless they are pairwise
  distinct almost surely and the probability distribution function, defined in \cref{lemma:tower}, $F(\hat\lambda^r) = \mathbb{P}\bigl( w^r_{\hat\lambda^r}(A_{n+1}) \ge 1-\gamma \,\bigm|\,
    A_1,\dots,A_n \bigr)$ is continuous. Then, \citep[Sec.~4.1]{angelopoulos2024conformalbook} implies that $F(\hat\lambda^r) \sim \mathrm{Beta}(k,\, n+1-k)$.
\end{proof}

\cref{thm:general guarantee} quantifies the calibration-set to calibration-set variability that the marginal guarantee averages over, and is the starting point for calibration-conditional guarantees \citep{vovk2012conditionalvalidityinductiveconformal}. In certain applications, the idealized residual field may induce a nonconformity score that cannot be computed directly. However, if this score admits a computable upper bound, the following corollary shows that calibrating on the upper bound preserves the desired coverage guarantee, although the resulting prediction band may be conservative.

\begin{corollary}[Calibration with a dominating score]\label{cor:dominating score}
  Let $r:\mathcal{A}\times X \rightarrow [0, \infty]$ be a residual field over $(X,\mathcal{X}, \pi)$ and
  $\sigma : \mathcal{A} \to [0,\infty]$ be Borel with $s^r(a) \le \sigma(a)$ for $\mu$-a.e.\ $a\in\mathcal{A}$. Calibrate on $\sigma$, i.e., set
  $\hat\lambda^\sigma$ as the $k$th smallest value among $\sigma(A_1),\dots,\sigma(A_n)$ if $k \le n$ and $\hat\lambda^\sigma := \infty$ otherwise. Then if $A_1,\dots,A_{n+1}$ are exchangeable,
  \[
    \mathbb{P}\bigl( w^r_{\hat\lambda^\sigma}(A_{n+1}) \ge 1-\gamma \bigr)
    \ge 1-\alpha .
  \]
\end{corollary}
\begin{proof}
  The variables $\sigma(A_1),\dots,\sigma(A_{n+1})$ are exchangeable, being the coordinatewise image of an exchangeable vector under the fixed Borel map $\sigma$, so exactly as in \cref{thm:general guarantee}, $\mathbb{P}\bigl( \sigma(A_{n+1} ) \le \hat\lambda^\sigma \bigr) \ge
    k/(n+1) \ge 1-\alpha$. Up to a $\mu^{\otimes(n+1)}$-null set, on this event $s^r(A_{n+1}) \le \sigma(A_{n+1}) \le \hat\lambda^\sigma$, and \cref{lemma:score-containment equivalence} implies that $w^r_{\hat\lambda^\sigma}(A_{n+1}) \ge 1-\gamma$. Measurability of the events follows from \cref{lemma:event-meas}.
\end{proof}

\begin{remark}[One-sidedness of the calibration]\label{remark:one-sided}
  If a computable quantity is guaranteed to overestimate the true score of $\mu$-almost everywhere input, then calibrating on it inherits the coverage guarantee
  for the true residual field. No upper bound accompanies \cref{cor:dominating score}, since the domination $s^r \le \sigma$ is one-sided, so the coverage may exceed $1-\alpha$ by more than $1/(n+1)$. This conservativeness is the price of calibrating on an alternative score.
\end{remark}

The results of this section depend only on the measurable structure of $\mathcal{A}$, the fact that $\pi$ is a probability measure on $(X,\mathcal{X})$, and the joint measurability of the residual field. They are otherwise agnostic to the models and error estimators used to construct that field, and therefore apply equally to non-FNO architectures, alternative pointwise error estimators, non-uniform grids, and evaluation spaces unrelated to the physical domain $D$. These results guarantee coverage but do not control the magnitude of $\hat\lambda^r$ or the width of the resulting uncertainty bands. Such efficiency depends on how informative the chosen residual field is, rather than on the conformal calibration argument itself.

\section{Implementation and Additional Results}\label{app:experiments}

\subsection{Implementation Details}\label{app:implem details}

\paragraph{Data.}
Both problems use the benchmark datasets released with the FNO \citep{li2021fourier}. For Darcy flow we use the $421 \times 421$ steady-state data (the same data used by \citealp{ma2024calibrated}), with constant forcing $f \equiv 1$ and zero Dirichlet boundary conditions. We discard the duplicated final row and column to obtain $N_{\max} = 420$ points per axis. For Navier-Stokes we use the $\nu = 10^{-4}$ vorticity data at $64 \times 64$ ($10\,000$ trajectories), together with the companion set of $20$ trajectories at $256 \times 256$ used as the fine reference grid, and learn the map from vorticity at $t = 11$ to $t = 21$. Working resolutions are obtained by strided subsampling of the reference grid, so every resolution sees the same functions. Each dataset is partitioned once, with a fixed seed (42), into four disjoint subsets (Table~\ref{tab:app-config}). The partition is drawn uniformly at random over
sample indices. The Navier–Stokes data are from the FNO release (\url{https://drive.google.com/drive/folders/1UnbQh2WWc6knEHbLn-ZaXrKUZhp7pjt-}, accessed August 2026), the Darcy data were obtained via the neuraloperator library's DarcyDataset.

\paragraph{Architecture and training.}
The prediction and error operators share the architecture of Table~\ref{tab:app-config} and differ only in their output activation: the error operator applies a final Softplus to ensure nonnegativity. The
prediction operator is trained with the relative $L^2$ loss from the Neural Operator library \citep{kossaifi2025librarylearningneuraloperators}, then frozen. The error operator is trained on its pointwise residual fields with the pinball loss at level $1 - \gamma = 0.9$. Both use AdamW with a constant learning rate and no schedule. Error targets are rescaled by a single constant computed from the error training set before fitting, for numerical conditioning. The constant
is fixed before calibration and absorbed into the calibrated factor, so it does not affect validity. Hyperparameters were not tuned per method or per calibration rule: a single configuration was chosen per problem so that all comparisons are between calibration rules applied to identical trained operators. Mode counts and epochs were chosen following standard FNO settings.

\begin{table}[h]
  \centering
  \caption{Configuration for both problems. Fourier modes are given per axis and are capped by the grid size at each resolution.}
  \label{tab:app-config}
  \begin{tabular}{lcc}
    \toprule
                         & Darcy flow                       & Navier--Stokes \\
    \midrule
    Resolutions $N$      & $12, 20, 30, 42, 60, 84$         & $16, 32, 64$   \\
    Fourier modes        & $10$ at $N{=}12$; $16$ otherwise & $16, 24, 32$   \\
    Hidden channels      & $32$                             & $32$           \\
    Fourier layers       & $4$                              & $4$            \\
    Lifting / projection & $64$ / $64$                      & $64$ / $64$    \\
    Optimiser            & AdamW                            & AdamW          \\
    Learning rate        & $10^{-3}$                        & $10^{-3}$      \\
    Weight decay         & $10^{-4}$                        & $10^{-4}$      \\
    Batch size           & $16$                             & $32$           \\
    Epochs               & $75$                             & $75$           \\
    $(\gamma, \alpha)$   & $(0.1, 0.1)$                     & $(0.1, 0.1)$   \\
    \midrule
    Prediction training  & $2000$                           & $5000$         \\
    Error training       & $1000$                           & $3000$         \\
    Calibration          & $500$                            & $1000$         \\
    Test                 & $1000$                           & $1000$         \\
    Resampling split     & $(500, 1000)$                    & $(500, 1500)$  \\
    \bottomrule
  \end{tabular}
\end{table}

\paragraph{Calibration and evaluation.}
Both calibration rules are applied to the same trained operators and the same calibration set, so all reported differences are attributable to the rules alone. For the resampled validity and efficiency experiments, the calibration and test sets are pooled and re-split $T = 3000$ times at the sizes given in Table~\ref{tab:app-config}, using identical permutations for both methods.

\paragraph{Compute.}
All training and evaluation was run on CPU on a single 2024 MacBook Pro (Apple M4 chip, 24GB memory). No GPU or cluster resources were used. Training one prediction/error pair takes approximately 5 minutes at the coarsest and 90 minutes at the finest resolution. Calibration and the $T = 3000$ resampling experiment take under 3 minutes in total. The fine resolution transfer test of Darcy flow, however took approximately 90 minutes to complete.

\subsection{Additional Numerical Results}\label{app:additional numerical}

This appendix complements \cref{sec:experiments} with per-resolution tables and distributional diagnostics. As in the main text, $N$ denotes resolution per spatial direction and $M=N^2$ denotes the number of grid points in 2D. On each fixed grid, the relevant validity statement is \cref{thm:grid guarantee}.

\begin{table}[!htbp]
  \centering
  \caption{Darcy flow: calibrated factor $\hat\lambda$ (single calibration split), mean bandwidth $\overline{B}$, and width ratio between methods at each resolution. Mean bandwidth is
    $\overline{B} := \frac{1}{n_{\mathrm{test}}}\sum_{i=1}^{n_{\mathrm{test}}}\frac{1}{M}\sum_{j=1}^{M}\hat\lambda\,\hat{\mathcal{E}}_M(A_i)(x_j)$.}
  \label{tab:darcy-efficiency}
  \vspace{0.3cm}
  \begin{tabular}{c cc| cc |c}
    \toprule
        & \multicolumn{2}{c}{Ours} & \multicolumn{2}{c}{Ma et al.} &                                            \\
    $N$ & $\hat\lambda$            & $\overline{B}$                & $\hat\lambda$ & $\overline{B}$      &
    ratio                                                                                                       \\
    \hline \rule{-1pt}{2.4ex}
    12  & 1.84                     & $9.43\times10^{-4}$           & 4.38          & $2.24\times10^{-3}$ & 2.37 \\
    20  & 1.59                     & $6.90\times10^{-4}$           & 2.24          & $9.71\times10^{-4}$ & 1.41 \\
    30  & 1.48                     & $5.52\times10^{-4}$           & 1.92          & $7.13\times10^{-4}$ & 1.29 \\
    42  & 1.64                     & $4.62\times10^{-4}$           & 2.05          & $5.77\times10^{-4}$ & 1.25 \\
    60  & 1.42                     & $5.41\times10^{-4}$           & 1.66          & $6.34\times10^{-4}$ & 1.17 \\
    84  & 1.54                     & $5.04\times10^{-4}$           & 1.79          & $5.87\times10^{-4}$ & 1.16 \\
    \hline
  \end{tabular}
\end{table}

Table~\ref{tab:darcy-efficiency} quantifies the efficiency gap from the main text on a single split. Our bands are uniformly tighter, from $2.37\times$ at $N=12$ to about $1.16$--$1.17\times$ at the finest resolutions. The correction term used by \cite{ma2024calibrated} decreases with resolution but remains conservative at practical $N$. The mean bandwidth $\overline{B}$ is not monotone in $N$ because operators are retrained independently at each resolution, so model quality and discretization both affect this quantity. By contrast, the per-resolution width ratio is smoother because both methods are applied to the same trained operators and residual fields.

\begin{figure}[!htbp]
  \hspace*{-1.3cm}
  \centering
  \includegraphics{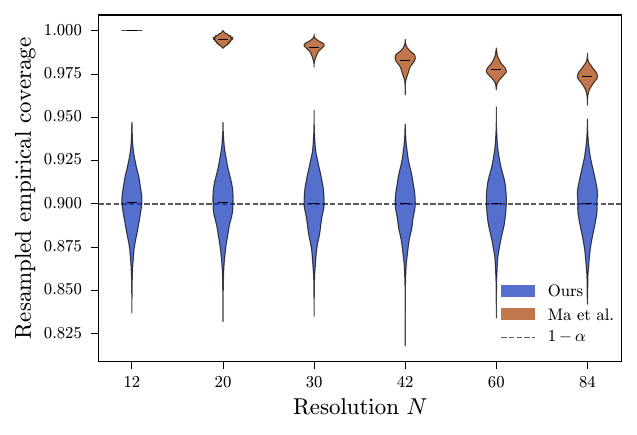}
  \caption{Darcy flow: distribution of re-sampled coverage by resolution and method, with means and target $1-\alpha$. Our distributions are centred near the target at each $N$, consistent with \cref{thm:grid guarantee}; those of \cite{ma2024calibrated} are shifted upward and collapse to coverage $1$ at the coarsest resolution.}
  \label{fig:darcy-violins}
\end{figure}

\begin{figure}[htbp]
  \hspace*{-1cm}
  \centering
  \includegraphics{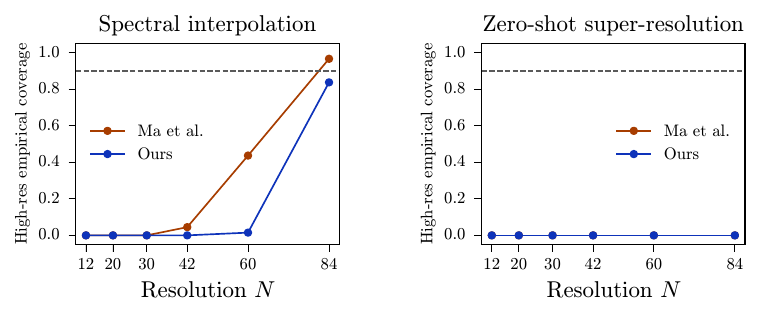}
  \caption{Darcy flow: coverage on the fine reference grid ($N_{\max}=420$) versus calibration resolution. \textbf{Left:} exact spectral interpolation of calibrated coarse-grid bands. \textbf{Right:} zero-shot super-resolution. This evaluates cross-grid transfer, which is outside the scope of \cref{thm:grid guarantee} (a fixed-grid guarantee). The correction in \cite{ma2024calibrated} increases width (up to $2.4\times$) yet reaches the target only at $N=84$, and zero-shot coverage is $0$ at all resolutions.}
  \label{fig:darcy-transfer}
\end{figure}

\begin{table}[!htbp]
  \centering
  \caption{Navier-Stokes: calibrated factor $\hat\lambda$, mean bandwidth $\overline{B}$, and the width ratio between the methods, per resolution.}
  \label{tab:ns-efficiency}
  \begin{tabular}{c cc |cc |c}
    \toprule
        & \multicolumn{2}{c}{Ours} & \multicolumn{2}{c}{Ma et al.} &                                       \\
    $N$ & $\hat\lambda$            & $\overline{B}$                & $\hat\lambda$ & $\overline{B}$ &
    ratio                                                                                                  \\
    \hline \rule{-1pt}{2.4ex}
    16  & 1.92                     & 0.419                         & 3.42          & 0.747          & 1.78 \\
    32  & 1.76                     & 0.354                         & 2.34          & 0.472          & 1.33 \\
    64  & 1.72                     & 0.354                         & 2.11          & 0.436          & 1.23 \\
    \hline
  \end{tabular}
\end{table}

Across both violin plots (\cref{fig:darcy-violins,fig:ns-violins}), our method is calibrated around the target across resolutions, while \citep{ma2024calibrated} is consistently above target. This is the expected efficiency-validity trade-off: both methods are valid, but the comparator is conservative.

\begin{figure}[htpb]
  \hspace*{-1cm}
  \centering
  \includegraphics{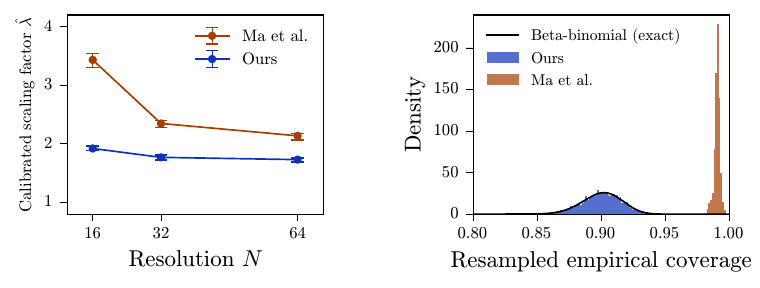}
  \caption{Navier--Stokes. \textbf{Left:} resampled median calibrated scaling factor $\hat{\lambda}$ versus resolution $N$, with empirical $5$th--$95$th percentile bars. \textbf{Right:} resampled coverage at $N=64$ against the Beta-Binomial reference implied by \cref{thm:grid guarantee} and the target $1-\alpha$.}
  \label{fig:ns-main}
\end{figure}

\begin{figure}[!htbp]
  \hspace{-1.3cm}
  \centering
  \includegraphics{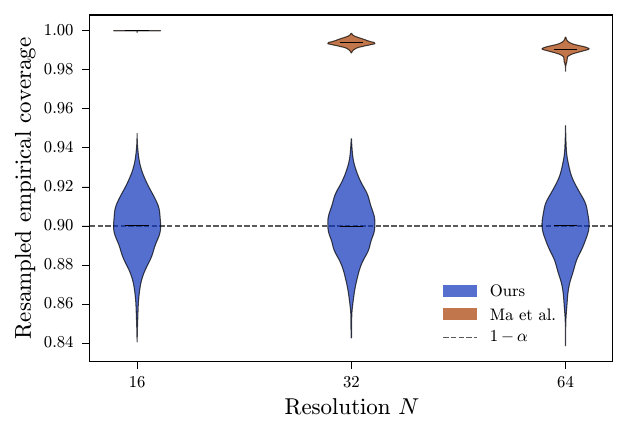}
  \caption{Navier--Stokes: distribution of re-sampled coverage by resolution and method, with means and target $1-\alpha$. Our distributions centre near the target, while those of \cite{ma2024calibrated} are shifted upward with smaller spread, indicating conservative validity.}
  \label{fig:ns-violins}
\end{figure}

For Navier--Stokes, \cref{fig:ns-main,tab:ns-efficiency} mirror the Darcy conclusions. Efficiency improves at every resolution (ratios $1.78\times$, $1.33\times$, $1.23\times$), and the re-sampled coverage for our method tracks the theorem-predicted Beta-Binomial reference closely. The distribution for \citep{ma2024calibrated} is shifted to higher coverage, confirming conservative validity with wider bands.

\end{document}